\documentclass[11pt,english]{article}
\usepackage[T1]{fontenc}
\usepackage[latin9]{inputenc}
\usepackage{geometry}
\usepackage{babel}
\usepackage{verbatim}
\usepackage{mathtools}
\usepackage{amsmath}
\usepackage{amsthm}
\usepackage{amssymb}
\usepackage{stmaryrd}
\usepackage[unicode=true,pdfusetitle,
 bookmarks=true,bookmarksnumbered=false,bookmarksopen=false,
 breaklinks=false,pdfborder={0 0 1},backref=false,colorlinks=false]
 {hyperref}

\makeatletter
\numberwithin{equation}{section}
\numberwithin{figure}{section}
\theoremstyle{plain}
\newtheorem{thm}{\protect\theoremname}
\theoremstyle{definition}
\newtheorem{defn}[thm]{\protect\definitionname}
\theoremstyle{remark}
\newtheorem{rem}[thm]{\protect\remarkname}
\theoremstyle{remark}
\newtheorem{notation}[thm]{\protect\notationname}
\theoremstyle{plain}
\newtheorem{prop}[thm]{\protect\propositionname}
\theoremstyle{definition}
\newtheorem{example}[thm]{\protect\examplename}
\theoremstyle{plain}
\newtheorem{lem}[thm]{\protect\lemmaname}
\theoremstyle{plain}
\newtheorem{cor}[thm]{\protect\corollaryname}

\date{}
\usepackage{tikz-cd}
\usepackage{wasysym}
\usepackage{MnSymbol}
\usepackage[tt=false]{libertine}
\usepackage[parfill]{parskip}
\usepackage{enumitem}
\setlist[itemize]{noitemsep,topsep=5pt}

\usepackage{titlesec}
\titleformat{\section}{\large\bfseries\filleft}{\thesection}{1em}{}[{\titlerule[0.8pt]}]

\renewcommand\labelenumi{(\roman{enumi})}
\renewcommand\theenumi\labelenumi

\DeclareMathOperator{\Spec}{Spec}

\DeclareMathOperator{\Fr}{Fr}

\DeclareMathOperator{\Tr}{Tr}
\DeclareMathOperator{\Id}{Id}

\DeclareMathOperator{\ev}{ev}

\DeclareMathOperator{\Conf}{Conf}
\DeclareMathOperator{\fudge}{fudge}

\DeclareMathOperator{\Res}{Res}
\DeclareMathOperator{\Nm}{Nm}
\DeclareMathOperator{\Aff}{Aff}
\DeclareMathOperator{\PConf}{PConf}
\DeclareMathOperator{\stab}{stab}
\DeclareMathOperator{\res}{res}

\let\oldtheorem\thm
\renewcommand{\thm}{\oldtheorem\normalfont}
\let\oldprop\prop
\renewcommand{\prop}{\oldprop\normalfont}
\let\oldcor\cor
\renewcommand{\cor}{\oldcor\normalfont}
\let\oldlem\lem
\renewcommand{\lem}{\oldlem\normalfont}
\newenvironment{ack}{\textit{Acknowledgements.}}{}

\tikzcdset{scale cd/.style={every label/.append style={scale=#1},
    cells={nodes={scale=#1}}}}

\makeatother

\usepackage[style=numeric,backend=bibtex8]{biblatex}
\providecommand{\corollaryname}{Corollary}
\providecommand{\definitionname}{Definition}
\providecommand{\examplename}{Example}
\providecommand{\lemmaname}{Lemma}
\providecommand{\notationname}{Notation}
\providecommand{\propositionname}{Proposition}
\providecommand{\remarkname}{Remark}
\providecommand{\theoremname}{Theorem}

\begin{document}
\global\long\def\A{\mathbb{A}}%

\global\long\def\C{\mathbb{C}}%

\global\long\def\E{\mathbb{E}}%

\global\long\def\F{\mathbb{F}}%

\global\long\def\G{\mathbb{G}}%

\global\long\def\H{\mathbb{H}}%

\global\long\def\N{\mathbb{N}}%

\global\long\def\P{\mathbb{P}}%

\global\long\def\Q{\mathbb{Q}}%

\global\long\def\R{\mathbb{R}}%

\global\long\def\O{\mathcal{O}}%

\global\long\def\Z{\mathbb{Z}}%

\global\long\def\ep{\varepsilon}%

\global\long\def\wangle#1{\left\langle #1\right\rangle }%

\global\long\def\ol#1{\overline{#1}}%

\global\long\def\mf#1{\mathfrak{#1}}%

\global\long\def\mc#1{\mathcal{#1}}%

\global\long\def\norm#1{\left\Vert #1\right\Vert }%

\global\long\def\et{\textup{ét}}%

\global\long\def\Et{\textup{Ét}}%

\title{A geometric approach to functional equations for general multiple
Dirichlet series over function fields}
\author{Matthew Hase-Liu\thanks{Department of Mathematics, Columbia University, New York, NY}\thanks{Email address: m.hase-liu@columbia.edu}}
\maketitle
\begin{abstract}
Sawin recently gave an axiomatic characterization of multiple Dirichlet
series over the function field $\F_{q}(T)$ and proved their existence
by exhibiting the coefficients as trace functions of specific perverse
sheaves. However, he did not prove that these series actually converge
anywhere, instead treating them as formal power series. 

In this paper, we prove that these series do converge in a certain
region, and moreover that the functions obtained by analytically continuing
them satisfy functional equations. 

For convergence, it suffices to obtain bounds on the coefficients,
for which we use the decomposition theorem for perverse sheaves, in
combination with the Kontsevich moduli space of stable maps to construct
a suitable compactification. 

For the functional equations, the key identity is a multi-variable
generalization of the relationship between a Dirichlet character and
its Fourier transform; in the multiple Dirichlet series setting, this
uses a density trick for simple perverse sheaves and an explicit formula
for intermediate extensions from the complement of a normal crossings
divisor. 
\end{abstract}
\tableofcontents{}

\section{Introduction}

A multiple Dirichlet series, roughly speaking, is a multi-variable
generalization of the well-studied single-variable Dirichlet series.
A usual (single-variable) Dirichlet series is a  series in a single
complex variable whose coefficients satisfy multiplicativity relations,
whereas a multiple Dirichlet series is a series in several complex
variables whose coefficients satisfy certain \textit{twisted} multiplicativity
relations instead.

The traditional perspective is to require the multiple Dirichlet series
to additionally satisfy a group of functional equations. Then, to
construct a specific multiple Dirichlet series, twisted multiplicativity
allows one to reduce to a local construction of coefficients of prime
powers, with the corresponding local generating functions satisfying
similar functional equations. Over the function field $\F_{q}(T)$,
\cite{chinta_mult} proved there is a local-to-global relationship
between these local generating functions and the global multiple Dirichlet
series. \cite{diaconu_pasol_moduli} moreover showed in the specific
setting of quadratic Dirichlet \textit{L}-series that this local-to-global
relationship actually uniquely characterizes the multiple Dirichlet
series. In particular, they gave an axiomatic characterization of
these specific multiple Dirichlet series (one of the axioms being
this local-to-global relationship), and in this setting, \cite{whitehead_multiple}
in his thesis was able to establish the functional equations using
this local-to-global relationship. 

Recently, \cite{sawin_general} massively generalized these ideas
by giving an axiomatic characterization of general multiple Dirichlet
series over $\F_{q}(T)$, and he proved the existence of such multiple
Dirichlet series by expressing the coefficients as trace functions
of a perverse sheaf. He also showed that these general multiple Dirichlet
series generalize many known examples of multiple Dirichlet series
that had appeared in the literature. 

Sawin, however, left open whether or not these series are genuine
analytic functions, as well as which functional equations are satisfied.
In this paper, we use geometric methods to answer these questions.

Fix $q$ to be a power of an odd prime. Let $\F_{q}[T]^{+}$ be the
set of monic single-variable polynomials over $\F_{q}$, and let $\mathcal{M}_{d}$
be the subset of degree $d$ monic polynomials. Fix a positive integer
$n$, and let $\chi\colon\F_{q}^{\times}\to\C^{\times}$ be a non-trivial
multiplicative character of order $n$. Let $m$ be a positive integer
and $M$ be a symmetric $m\times m$ matrix with integer entries (modulo
$n$).

Sawin constructs a multiple Dirichlet series
\[
L\left(u_{1},\ldots,u_{m};M\right)=\sum_{f_{1},\ldots,f_{m}\in\F_{q}[t]^{+}}a\left(f_{1},\ldots,f_{m};M\right)u_{1}^{\deg f_{1}}\cdots u_{m}^{\deg f_{m}},
\]
where each ``$a$-coefficient'' $a\left(f_{1},\ldots,f_{m};M\right)$
arises as the trace of Frobenius acting on the stalk (at the tuple
$\left(f_{1},\ldots,f_{m}\right)$ viewed as an element of the moduli
space of tuples of monic polynomials $\prod_{i=1}^{m}\A_{\F_{q}}^{d_{i}}$)
of a certain perverse sheaf defined by the parameters $M$ and the
degrees of the polynomials $f_{1},\ldots,f_{m}$. 

Our first main result is as follows:
\begin{thm}
\label{thm:firstmain}Let $n,m,$ and $s$ be positive integers, and
$M$ a symmetric $m\times m$ matrix with coefficients in $\Z/n\Z$.
Then, $L\left(u_{1},\ldots,u_{m};M\right)$ is an analytic function
with a non-empty region of convergence.
\end{thm}

When $M_{1,1}=\cdots=M_{1,s}=0$ and $M_{1,s+1},\ldots,M_{1,m}\ne0$,
the functional equations relate the series $L\left(u_{1},\ldots,u_{m};M\right)$
to another series, which is modified in two ways: 1. some fudge factors
are added to the coefficients, and 2. the matrix $M$ is replaced
with another matrix $M'$. 

Let $\psi$ be the non-trivial additive character $e^{2\pi i\Tr_{\F_{q}/\F_{p}}(-)/p}$
on $\F_{q}$, and let $G\left(\chi,\psi\right)=\sum_{x\in\F_{q}^{\times}}\chi(x)\psi(x)$. 
\begin{defn}
For non-negative integers $d_{s+1},\ldots,d_{m}\ge0$ and a list of
integers $M_{1,s+1},\ldots,M_{1,m}$, let 
\[
\fudge\left(d_{s+1},\ldots,d_{m};M_{1,s+1},\ldots,M_{1,m}\right)\coloneqq\frac{\chi(-1)^{\sum_{s+1\le i<j\le m}d_{i}d_{j}M_{1,i}}(-1)^{\sum_{i\ge s+1}\frac{d_{i}(d_{i}-1)(q-1)}{4}}}{\prod_{i\ge s+1}G\left(\chi^{M_{1,i}},\psi\right)^{d_{i}}}.
\]
\end{defn}

\begin{defn}
Let $M$ be a symmetric $m\times m$ matrix with coefficients in $\Z/n\Z$
for $n$ even.  Define $M'$ to be another symmetric $m\times m$
matrix such that 
\begin{enumerate}
\item $M'_{i,j}=M_{i,j}+M_{1,i}+M_{1,j}$ for $j>i\ge s+1,$ 
\item $M'_{i,i}=M_{i,i}+M_{1,i}+n/2$ for $i\ge s+1$, 
\item $M'_{1,i}=-M_{1,i}$ for all $i$,
\item $M_{i,j}'=M_{i,j}$ for $j>i>1$ and $i\le s$, and 
\item $M_{i,i}'=M_{i,i}$ for $i\le s$.
\end{enumerate}
\end{defn}

\begin{rem}
One should think of $a$-coefficients with matrix $M'$ (in comparison
to $a$-coefficients with matrix $M$) as playing a similar role to
the conjugate of a Dirichlet character. In fact, note that $\left(M'\right)'=M$.
\end{rem}

Define $L_{\fudge}\left(u_{1},\ldots,u_{m};M\right)$ to be a slight
variant of $L\left(u_{1},\ldots,u_{m};M\right)$:
\begin{align*}
 & L_{\fudge}\left(u_{1},\ldots,u_{m};M\right)\\
 & \coloneqq\sum_{f_{1},\ldots,f_{s}\in\F_{q}[t]^{+}}\sum_{d_{s+1},\ldots,d_{m}}b\left(d_{\ge s+1};M_{1,\ge s+1}\right)\sum_{f_{s+1}\in\mathcal{M}_{d_{s+1}},\ldots,f_{m}\in\mathcal{M}_{d_{m}}}a\left(f_{1},\ldots,f_{m};M'\right)u_{1}^{d_{1}}\cdots u_{m}^{d_{m}},
\end{align*}
where 
\[
b\left(d_{\ge s+1};M_{1,\ge s+1}\right)=\begin{cases}
\frac{1}{q^{\sum_{i=s+1}^{m}d_{i}/2}\fudge\left(d_{\ge s+1};M_{1,\ge s+1}\right)} & \text{if }n\text{ divides }\sum_{i=s+1}^{m}d_{i}M_{1,i},\\
\frac{\chi(-1)^{\sum_{i=s+1}^{m}d_{i}M_{1,i}}G\left(\chi^{\sum_{i=s+1}^{m}d_{i}M_{1,i}},\psi\right)}{q^{1+\sum_{i=s+1}^{m}d_{i}/2}\fudge\left(d_{\ge s+1};M_{1,\ge s+1}\right)} & \text{else}.
\end{cases}
\]

\begin{rem}
In the special case $n=2$ where $\chi$ is a quadratic character
and $q\equiv1\bmod4$, we have $M'=M$ and $L_{\fudge}$ is simply
a change of variables of $L$. 
\end{rem}

Also, note that permuting the variables of $L\left(u_{1},\ldots,u_{m};M\right)$
is equivalent to permuting the entries of $M$ up to adding a sign
into the coefficients of the series, so it suffices to consider functional
equations in the first variable. 

Our second main result is as follows:
\begin{thm}
\label{thm:mainresult}Let $n$ be a positive even integer, $m$ and
$s$ be positive integers, and $M$ a symmetric $m\times m$ matrix
with coefficients in $\Z/n\Z$. Without loss of generality, assume
$M_{1,1}=\cdots=M_{1,s}=0$ and $M_{1,s+1},\ldots,M_{1,m}$ are not
zero, with $s\ge1$. Let $\zeta_{n}$ be a primitive complex $n$th
root of unity. Then, using the notation above, we have the functional
equation 
\begin{align*}
 & u_{1}\left(qu_{1}-1\right)L\left(u_{1},\ldots,u_{m};M\right)\\
 & =\left(qu_{1}-1\right)L_{\fudge}\left(\frac{1}{qu_{1}},u_{2},\ldots u_{s},q^{1/2}u_{1}u_{s+1},\ldots,q^{1/2}u_{1}u_{m};M\right)\\
 & \qquad-\frac{qu_{1}+u_{1}-2}{n}\sum_{0\le j\le n-1}L_{\fudge}\left(\frac{1}{qu_{1}},u_{2},\ldots,u_{s},\zeta_{n}^{jM_{1,s+1}}q^{1/2}u_{1}u_{s+1},\ldots,\zeta_{n}^{jM_{1,m}}q^{1/2}u_{1}u_{m};M\right),
\end{align*}
which is an equality of analytic functions on a domain including the
region of convergence of both sides.
\end{thm}

\begin{rem}
At first glance, this does not obviously look like a functional equation
because the left-hand side involves $L$ whereas the right-hand side
involves $L_{\fudge}.$ This is due to $\fudge\left(d_{\ge s+1};M_{1,\ge s+1}\right)$
being hard to separate in such a way that its factors are absorbed
into the change of variables in the functional equation. One way to
rectify this is by defining a refined version of $L$ including congruence
conditions on the powers of the $u_{i}$. This version of the functional
equation relating multiple $L$ with congruence conditions to $L$
with a change of variables and possibly a different set of congruence
conditions is easy to extract from the theorem above (but requires
introducing new notation that will not be used again).
\end{rem}

\begin{rem}
The assumption on the parity of $n$ is mainly due to the fact that
the unique multiplicative character $\F_{q}^{\times}\to\C^{\times}$
of order two is simply $\chi^{n/2}$. Note that if $\chi$ has odd
order and $q$ is not a power of 2, we can replace $\chi$ with its
square root, double $n$, and double the entries of $M$ to get an
equivalent series where $n$ is even.
\end{rem}

\begin{rem}
We are not able to prove that $L\left(u_{1},\ldots,u_{m};M\right)$
has meromorphic continuation to $\C^{m}$ in general, since we only
have a subset of the group of possible functional equations (even
then, it may not be possible). For this reason, we do not attempt
to optimize the region for which $L\left(u_{1},\ldots,u_{m};M\right)$
has meromorphic continuation, i.e. such as using tools like Bochner's
tube theorem (c.f. \cite{BFG}).
\end{rem}

We prove the main results in two steps. The first step is to establish
the functional equation as a formal equality of power series. The
second step is to prove both sides of the functional equation converge
in regions of $\C^{m}$, say, $R_{1}$ and $R_{2}$, and to verify
that $R_{1}\cap R_{2}$ is not empty. In particular, we prove Theorem
\ref{thm:firstmain} in the second step and Theorem \ref{thm:mainresult}
as a combination of both steps.

The first step generalizes Proposition 4.4 of \cite{sawin_general},
which explicitly computes the $a$-coefficient when $M$ is of the
form $\begin{bmatrix}0 & -1\\
-1 & n/2
\end{bmatrix}$. The main trick is to establish a multi-variable variant of the relationship
between the conjugate of a Dirichlet character and the Fourier transform
of the Dirichlet character. Using this input, obtaining the functional
equation is mostly a formal, though tedious, calculation, that closely
mirrors that of the single-variable setting. 

To establish this connection with the Fourier transform in his example,
Sawin uses a density argument, where he first proves the claim easily
for tuples of monic polynomials that are square-free and pair-wise
relatively prime, and then uses properties of perverse sheaves to
extend the result to all tuples. 

In our more general setting, the claim, even for tuples of monic polynomials
that are square-free and pair-wise relatively prime, is not obvious,
and we use another geometric idea: The intermediate extension (which
is typically a complex in the derived category) of a tame lisse sheaf
with finite monodromy on the complement of a normal crossings divisor
has an explicit description as a genuine sheaf. Using this, we obtain
an explicit formula for the $a$-coefficients for a slightly larger
set of tuples. Then, using a boot-strapping argument, we obtain the
relationship between an $a$-coefficient with matrix $M'$ (the analogous
notion of a conjugate Dirichlet character) and the Fourier transform
of an $a$-coefficient with matrix $M$.

The second step is mainly about bounding individual $a$-coefficients
for fixed tuples $\left(f_{1},\ldots,f_{m}\right)$ and sums of $a$-coefficients
over all tuples $\left(f_{1},\dots,f_{m}\right)$ of fixed degrees.
Indeed, with these bounds, obtaining the regions of convergence follows
from elementary analysis of power series. The general strategy for
bounding comes from another geometric idea, namely by compactifying
the space of tuples of monic polynomials that are square-free and
pair-wise relatively prime using a quotient of the Kontsevich moduli
stack $\mathcal{\ol{\mathcal{M}}}_{0,r}\left(\P^{1},1\right)$---the
space of stable maps from a genus zero curve to $\P^{1}$ of degree
one---by a Young subgroup of the symmetric group $S_{r}$. The decomposition
theorem for perverse sheaves allows us to bound the $a$-coefficients
in terms of cohomology of this compactification, which, after translating
to the situation over characteristic zero, can be done by bounding
the number of cells using a combinatorial argument related to counting
rooted planar trees. 

Finally, unless otherwise stated, a variety is an irreducible, reduced,
separated scheme of finite type over a field.

\begin{ack} First and foremost, I would like to thank my advisor
Will Sawin for suggesting this problem; I'm immensely grateful to
Will for his tremendous and invaluable guidance, both high-level and
technical,  throughout every step of this project. I would also like
to thank Adrian Diaconu, Anh Trong Nam Hoang, Donggun Lee, Takyiu
Liu, Amal Mattoo, Che Shen, Fan Zhou, and the anonymous reviewer for
helpful discussions. 

The author was partially supported by National Science Foundation
Grant Number DGE-2036197.\end{ack}

\section{Preliminaries}

The section comprises a collection of four subsections that recall
and prove some technical results that may be of independent interest. 

The first subsection gives a quick review of the needed function field
number theory and explanation of Sawin's general construction of multiple
Dirichlet series. 

The second subsection explains in detail an explicit formula for a
specific intermediate extension. A priori, this is an abstract object
in the derived category of $\ell$-adic sheaves, but the specific
situation we work in allows us to view the intermediate extension
as a genuine sheaf. This formula is crucial in both steps of establishing
the functional equation. For the first step, the formula is one of
the key ingredients in proving the relationship between $a$-coefficients
and Fourier transforms. For the second step, the formula, in combination
with the decomposition theorem, is the key tool to relate $a$-coefficients
to cohomology of lisse sheaves on the Kontsevich moduli space $\ol{\mathcal{M}}_{0,r}\left(\P^{1},1\right)$.

The third subsection is about proving bounds on the counts of different
kinds of trees. 

The fourth subsection gives bounds for the cohomology of arbitrary
lisse sheaves on the Kontsevich moduli space $\ol{\mathcal{M}}_{0,r}\left(\P^{1},1\right)$
quotiented by a Young subgroup. The main input is the bounds on tree
counts from the previous subsection.

\subsection{\label{subsec:notation}Notation and Sawin's construction}

We first establish notation from function field analytic number theory
and for the rest of the paper, recalling some from the introduction
(for more details, look at \cite{sawin_general}):
\begin{enumerate}
\item $q$ is a fixed power of an odd prime $p$.
\item $\ell$ is a fixed prime not equal to $p$, and we fix an isomorphism
$\ol{\Q_{\ell}}\cong\C$.
\item $\F_{q}[T]$ is the ring of polynomials in one variable over $\F_{q}$. 
\item $\F_{q}[T]^{+}$ is the subset of monic single-variable polynomials
over $\F_{q}$.
\item $\mathcal{M}_{t}$ is the subset of monic single-variable polynomials
over $\F_{q}$ of degree $t$.
\item $\mathcal{P}_{<t}$ is the subset of single-variable polynomials over
$\F_{q}$ of degree less than $t$.
\item $m$ and $n$ are fixed positive integers.
\item $M$ is a symmetric $m\times m$ matrix with entries in $\Z/n\Z$.
\item $\chi\colon\F_{q}^{\times}\to\C^{\times}$ is a non-trivial multiplicative
character of order $n$.
\item $\psi\colon\F_{q}\to\C^{\times}$ is the non-trivial additive character
$e^{2\pi i\frac{\Tr_{\F_{q}/\F_{p}}(-)}{p}}.$
\item $G\left(\chi,\psi\right)$ is the Gauss sum $\sum_{x\in\F_{q}^{\times}}\chi(x)\psi(x).$
\item For $e$ a positive integer, $\chi_{e}\colon\F_{q^{e}}^{\times}\to\C^{\times}$
is the multiplicative character given by the composition $\F_{q^{e}}^{\times}\overset{\Nm}{\to}\F_{q}^{\times}\overset{\chi}{\to}\C^{\times}$,
where $\Nm$ is the norm map.
\item The resultant $\Res\left(f,g\right)$ of $f,g\in\F_{q}[T]$ is defined
as the product of values of $f$ at the roots of $g$. In particular,
$\Res\left(f,g\right)=0$ iff $f$ and $g$ share a common root.
\item The residue symbol $\left(\frac{f}{g}\right)_{\chi}$ is defined as
$\chi\left(\Res\left(f,g\right)\right)$. 
\item The residue $\res(f)$ of a rational function $f$ is defined as the
coefficient of $T^{-1}$ when $f$ is written as a Laurent series.
\item $e\left(-\right)$ is the composition $\psi\left(\res\left(-\right)\right).$
\end{enumerate}
Note that affine space $\A_{\F_{q}}^{d}$ can be viewed as a moduli
space for monic polynomials of degree $d$. Indeed, for an $\F_{q}$-algebra
$R$, $\A_{\F_{q}}^{d}\left(R\right)=R^{d}=\left\{ \left(r_{d-1},\ldots,r_{0}\right):r_{i}\in R\right\} $,
which we identify with the set of monic single-variable polynomials
over $R$ of degree $d$: $\left\{ t^{d}+r_{d-1}t^{d-1}+\cdots+r_{0}:r_{i}\in R\right\} $. 

Consequently, for non-negative integers $d_{1},\ldots,d_{m}$, we
can view $\prod_{i=1}^{m}\A_{\F_{q}}^{d_{i}}$ as a moduli space for
tuples of monic polynomials of fixed degrees $d_{1},\ldots,d_{m}$. 

Then, for such $d_{1},\ldots,d_{m}$, define the polynomial function
\[
F_{d_{1},\ldots,d_{m}}=\prod_{i=1}^{m}\Res\left(f_{i}',f_{i}\right)^{M_{i,i}}\prod_{1\le i<j\le r}\Res\left(f_{i},f_{j}\right)^{M_{i,j}}
\]
on $\prod_{i=1}^{m}\A^{d_{i}}$. 

Let $U$ be the open subset for which $F_{d_{1},\ldots,d_{m}}$ is
invertible. Geometrically, $F_{d_{1},\ldots,d_{m}}$ defines a morphism
$\prod_{i=1}^{m}\A^{d_{i}}\to\A^{1}$, and $U$ is simply the preimage
of $\G_{m}\subset\A^{1}$. In particular, by abuse of notation, $F_{d_{1},\ldots,d_{m}}$
also defines a morphism $U\to\G_{m}$. 

On $\G_{m}$, we have a Kummer sheaf associated to $\chi$, which
is perhaps best understood as a one-dimensional representation of
the etale fundamental group of $\G_{m}$. Namely, there is a natural
surjection $\pi_{1}\left(\G_{m}\right)\twoheadrightarrow\F_{q}^{\times}$
arising from Kummer theory, and the composition $\pi_{1}\left(\G_{m}\right)\twoheadrightarrow\F_{q}^{\times}\overset{\ol{\chi}}{\to}\C^{\times}$
gives a continuous one-dimensional representation of $\pi_{1}\left(\G_{m}\right)$.
By the correspondence between such representations and lisse rank
one etale sheaves, we obtain the Kummer sheaf $\mathcal{L}_{\chi}$
on $\G_{m}$. 

Using $F_{d_{1},\ldots,d_{m}}\colon U\to\G_{m}$, we can pull back
$\mathcal{L}_{\chi}$ to $U$, which we denote by $\mathcal{L}_{\chi}\left(F_{d_{1},\ldots,d_{m}}\right)$. 

Recall that there is an abelian category of ``perverse sheaves''
inside the derived category of $\ell$-adic sheaves, which is given
by the heart of a certain \textit{t}-structure, c.f. \cite{bbdg}.
There are two important examples of perverse sheaves that will appear
in this paper:
\begin{enumerate}
\item If $X$ is a smooth variety and $L$ is a lisse sheaf on $X$, then
$L\left[\dim X\right]$ is perverse.
\item If $X$ is a variety, $j\colon U\subset X$ is the inclusion of an
open subset, and $A$ is a perverse sheaf on $U$, then there is a
perverse sheaf $j_{!*}A$ on $X$, called the intermediate extension
of $A$, defined as the unique extension of $A$ that has no non-trivial
sub-objects or quotients supported on $X\backslash U$. 
\end{enumerate}
Combining these two examples, $j_{!*}\left(\mathcal{L}_{\chi}\left(F_{d_{1},\ldots,d_{m}}\right)\left[d_{1}+\cdots+d_{m}\right]\right)$
is a perverse sheaf on $\prod_{i=1}^{m}\A^{d_{i}}$, and Sawin defines
\[
K_{d_{1},\ldots,d_{m}}=j_{!*}\left(\mathcal{L}_{\chi}\left(F_{d_{1},\ldots,d_{m}}\right)\left[d_{1}+\cdots+d_{m}\right]\right)\left[-d_{1}-\cdots-d_{m}\right],
\]
shifted up so that generically $K_{d_{1},\ldots,d_{m}}$ agrees with
$\mathcal{L}_{\chi}\left(F_{d_{1},\ldots,d_{m}}\right)$.

Then, for a tuple $\left(f_{1},\ldots,f_{m}\right)\in\prod_{i=1}^{m}\A^{d_{i}}$,
Sawin defines the $a$-coefficient to be 
\begin{equation}
\widetilde{a}\left(f_{1},\ldots,f_{m};q,\chi,M\right)=\Tr\left(\Fr_{q},\left(K_{d_{1},\ldots,d_{m}}\right)_{\left(f_{1},\ldots,f_{m}\right)}\right),\label{eq:defofacoeffviaperverse}
\end{equation}
where $\Fr_{q}$ is the geometric Frobenius. 
\begin{notation}
Note that the $a$-coefficient is a function of $f_{1},\ldots,f_{m},q,\chi,$
and $M$; however, when $q,\chi,$ and/or $M$ are clear from context,
we will drop them. 
\end{notation}

Before recalling the main theorem of \cite{sawin_general}, we review
what Sawin calls ``compatible systems of sets of ordered pairs of
Weil numbers and integers.'' 

A Weil number is an algebraic number $\alpha$ such that there exists
some $i$ such that for any embedding $\ol{\Q_{\ell}}$ into $\C$,
the absolute value of the image of $\alpha$ is $q^{i/2}$. 

A set of ordered pairs of Weil numbers and integers is simply a set
of ordered pairs $\left(\alpha_{j},c_{j}\right)$ indexed by $j$,
such that $\alpha_{j}$ is a Weil number, $c_{j}$ is a non-zero integer,
and $\alpha_{j}\ne\alpha_{j'}$ for $j\ne j'$. 

A function $\gamma\left(-,-\right)$ from pairs of a prime power $q$
(of a fixed prime $p$) and a multiplicative character $\chi\colon\F_{q}^{\times}\to\C^{\times}$
to Weil numbers is said to be a compatible system of Weil numbers
if 
\[
\gamma\left(q^{e},\chi_{e}\right)=\gamma\left(q,\chi\right)^{e}
\]

A function $J\left(-,-\right)$ from pairs of a prime power $q$ (of
a fixed prime $p$) and a multiplicative character $\chi\colon\F_{q}^{\times}\to\C^{\times}$
to ordered pairs of Weil numbers and integers is said to be a compatible
system of sets of ordered pairs of Weil numbers and integers if $J\left(q,\chi\right)=\left\{ \left(\alpha_{j},c_{j}\right)\right\} $
means $J\left(q^{e},\chi_{e}\right)=\left\{ \left(\alpha_{j}^{e},c_{j}\right)\right\} $.
In the theorem below, we will introduce a function $J$ depending
on parameters $d_{1},\ldots,d_{m},q,\chi,$ and $M$; when $q$ and
$\chi$ vary, $J$ is a compatible system of sets of ordered pairs
of Weil numbers and integers.

The main theorem of \cite{sawin_general} gives an axiomatic characterization
of $a$-coefficients:
\begin{thm}
\label{thm:axioms}Using the notation of this subsection, the data
of $a\left(\underset{m}{\underbrace{-,\ldots,-}};-,-,M\right)$ a
complex-valued function on tuples of monic polynomials $\left(f_{1},\ldots,f_{m}\right)$
and a pair of a prime power $q$ and a multiplicative character $\chi$,
along with $J\left(\underset{m}{\underbrace{-,\ldots,-}};-,-,M\right)$
a function from tuples of non-negative integers $\left(d_{1},\ldots,d_{m}\right)$
to compatible systems of sets of ordered pairs of Weil numbers and
integers uniquely satisfies the following axioms:
\begin{enumerate}
\item (Twisted multiplicativity) If $\prod_{i=1}^{m}f_{i}$ and $\prod_{i=1}^{m}g_{i}$
are relatively prime, then 
\begin{align*}
 & a\left(f_{1}g_{1},\ldots,f_{m}g_{m};M\right)\\
 & =a\left(f_{1},\ldots,f_{m};M\right)a\left(g_{1},\ldots,g_{m};M\right)\prod_{1\le i\le m}\left(\frac{f_{i}}{g_{i}}\right)_{\chi}^{M_{i,i}}\left(\frac{g_{i}}{f_{i}}\right)_{\chi}^{M_{i,i}}\prod_{1\le i<j\le m}\left(\frac{f_{i}}{g_{j}}\right)_{\chi}^{M_{i,j}}\left(\frac{g_{i}}{f_{j}}\right)_{\chi}^{M_{i,j}}.
\end{align*}
 
\item $a\left(1,\ldots,1;M\right)=a\left(1,\ldots,1,f,1,\ldots,1;M\right)=1$
for all linear polynomials $f$.
\item 
\[
a\left(\pi^{e_{1}},\ldots,\pi^{e_{m}};M\right)=\left(\frac{\pi'}{\pi}\right)_{\chi}^{\sum_{i=1}^{m}e_{i}M_{i,i}}\sum_{j\in J\left(e_{1},\ldots,e_{m};q,\chi,M\right)}c_{j}\alpha_{j}^{\deg\pi}
\]
for a prime $\pi$.
\item 
\[
\sum_{f_{1}\in\mathcal{M}_{d_{1}},\ldots,f_{m}\in\mathcal{M}_{d_{m}}}a\left(f_{1},\ldots,f_{m};M\right)=\sum_{j\in J\left(d_{1},\ldots,d_{m};q,\chi,M\right)}c_{j}\frac{q^{\sum_{i=1}^{m}d_{i}}}{\ol{\alpha_{j}}}.
\]
\item $\left|\alpha_{j}\right|\le q^{\frac{\sum_{i=1}^{m}d_{i}}{2}-1}$
for $\sum_{i=1}^{m}d_{i}\ge2$. 
\end{enumerate}
\end{thm}

Existence is given by setting $a\left(f_{1},\ldots,f_{m};q,\chi,M\right)=\widetilde{a}\left(f_{1},\ldots,f_{m};q,\chi,M\right)$
using (\ref{eq:defofacoeffviaperverse}), and letting $J\left(d_{1},\ldots,d_{m};q,\chi,M\right)$
be the set of ordered pairs of eigenvalues of $\Fr_{q}$ acting on
$\left(K_{d_{1},\ldots,d_{m}}\right)_{\left(t^{d_{1}},\ldots,t^{d_{m}}\right)}$,
counted with signed multiplicities (the sign comes from the fact that
$\left(K_{d_{1},\ldots,d_{m}}\right)_{\left(t^{d_{1}},\ldots,t^{d_{m}}\right)}$
is a complex). 
\begin{notation}
As a consequence of Theorem \ref{thm:axioms}, we will use $a\left(f_{1},\ldots,f_{m};q,\chi,M\right)$
to denote both the axiomatic and geometric characterization of $a$-coefficients.
\end{notation}

\subsection{An explicit formula for a particular intermediate extension }
\begin{prop}
\label{prop:IC_is_sheaf}Let $X$ be a smooth variety, $D\hookrightarrow X$
a normal crossings divisor, $U=X\backslash D$, $j\colon U\subset X$
the open immersion, and $L$ a tame lisse sheaf on $U$ with finite
monodromy. 

Suppose we write $j$ as $j'\circ j''$, where $j''$ is the inclusion
of $U$ into the complement of the divisors for which $L$ has trivial
local monodromy, and $j'$ is the remaining inclusion into $X$. 

Then, $j''_{*}L$ is a lisse sheaf and 
\[
j_{!*}\left(L[d]\right)[-d]=j'_{!}j''_{*}L,
\]
where $d$ is the dimension of $X$. 
\begin{proof}
Let us first show that we actually have 
\[
j_{!*}\left(L[d]\right)[-d]=j_{*}L.
\]

Etale-locally, we may assume $D$ is of the form $t_{1}\cdots t_{s}=0$.
Let $X'$ be the closed subscheme in $X\times\A^{s}$ cut out by $\left(t'_{i}\right)^{N}=t_{i}$
for sufficiently large $N$. 

Consider the following Cartesian diagram:\[\begin{tikzcd} 	{U'} & {X'} \\ 	U & X 	\arrow["j"', hook, from=2-1, to=2-2] 	\arrow["{\pi'}"', from=1-1, to=2-1] 	\arrow["\pi", from=1-2, to=2-2] 	\arrow["{j'}", hook, from=1-1, to=1-2] \ar["\square", phantom, from=1-1, to=2-2] \end{tikzcd}\]Here,
$\pi$ is the projection onto $X,$ $j$ is the inclusion $U\subset X$,
and $\pi'$ and $j'$ are the respective pull-backs.

By Abhyankar's lemma (c.f. \cite{Grothendieck_Raynaud}), $\pi^{*}L$
is trivial, $\pi$ is finite, and $\pi'$ is etale. Moreover, the
monodromy representation of $\pi^{*}L$ factors through an abelian
product group $G$, which implies that the representation is the tensor
product of one-dimensional representations. 

Each of these one-dimensional representations can be realized as monodromy
representations associated to Kummer sheaves on $\G_{m}$ because
the direct factors of $G$ all look like groups of roots of unity
(more precisely inverse limits over $n$ prime to $p$ of $\mu_{n}$). 

Since intermediate extension and pushforwards commute with box products,
we can assume that $L$ is a lisse rank one sheaf on $\G_{m}$ that
has non-trivial local monodromy around the single closed point $\{p\}=\A^{1}\backslash\G_{m}$. 

Hence, by definition of the intermediate extension, we have $j_{!*}\left(L[1]\right)=j_{*}L[1]$,
from which it follows that 
\[
j_{!*}\left(L[d]\right)[-d]=j_{*}L=j'_{*}j''_{*}L.
\]

Next, let us show that the natural map $j'_{!}\left(j''_{*}L\right)\to j'_{*}\left(j''_{*}L\right)$
is an isomorphism. It suffices to do this on the level of stalks.

Recall that the local monodromy around an irreducible component is
obtained by taking a geometric generic point $\ol{\eta}$, and then
pulling $L$ back to $\Spec$ of the fraction field $K_{\ol{\eta}}^{\text{sh}}$
of the strictly Henselian local ring $\mathcal{O}_{X,\ol{\eta}}^{\text{sh}}$,
i.e. the strict localization $\widetilde{X}_{\ol{\eta}}$. 

By Tag 03Q7 of \cite{stacks}, the stalk at $p$ of $j_{*}L$ is given
by 
\begin{align*}
H^{0}\left(\widetilde{\A^{1}}_{\ol p}\times_{\A^{1}}\G_{m},L\right) & =H^{0}\left(\Spec K_{\ol p}^{\text{sh}},L\right)=\left(L_{\ol p}\right)^{G_{K_{\ol p}^{\text{sh}}}},
\end{align*}
where $G_{K_{\ol p}^{\text{sh}}}$ is the absolute Galois group of
$K_{\ol p}^{\text{sh}}.$ But this is zero for $\ol p$ such that
the local monodromy is non-trivial.

Finally, we show that $j''_{*}L$ is lisse. For the rest of the argument,
we may assume $j=j''$ by removing all divisors around which the local
monodromy is non-trivial. 

Suppose the divisors are $D_{1},,\ldots,D_{s}$. Let $g$ be the open
immersion $U\subset X\backslash\left\{ D_{2}\cup\cdots\cup D_{s}\right\} $.
Let us first show that $g_{*}L$ is lisse. Suppose $\ol{\eta}$ is
a geometric generic point of $D_{1}$. Since $g$ is an open immersion,
pulling $g_{*}L$ back to $\widetilde{X}_{\ol{\eta}}$ is the same
as pulling back $L$ to $\Spec K_{\ol{\eta}}^{\text{sh}}=\widetilde{X}_{\ol{\eta}}-\ol{\eta}$,
then pushing it forward to $\widetilde{X}_{\ol{\eta}}$. By assumption,
this is simply the pushforward of $\ol{\Q_{\ell}}$, so it follows
that $g_{*}L$ is lisse on a neighborhood of $\ol{\eta}$. Let $D_{1}'$
be the complement of the open locus where $g_{*}L$ is lisse. Since
$D_{1}'$ does not contain the generic point, it has codimension at
least two.

For the sake of contradiction, suppose $D_{1}'\ne\emptyset$. Then,
let $\ol{\eta'}$ be a geometric generic point of $D_{1}'$. Again,
pulling $g_{*}L$ back to $\widetilde{X}_{\ol{\eta'}}$ is the same
as pulling $L$ back to $\widetilde{X}_{\ol{\eta'}}\times_{X}U=\widetilde{X}_{\ol{\eta'}}-\ol{\eta'}$,
and then pushing forward to $\widetilde{X}_{\ol{\eta'}}$. By Grothendieck's
version of Zariski--Nagata purity (c.f. Tag 0BMA of \cite{stacks}),
we know $\pi_{1}\left(\widetilde{X}_{\ol{\eta'}}-\ol{\eta'}\right)=\pi_{1}\left(\widetilde{X}_{\ol{\eta'}}\right)=\pi_{1}\left(\ol{\eta'}\right)$,
which is trivial; indeed, $\widetilde{X}_{\ol{\eta'}}-\ol{\eta'}$
is an open subset of $\widetilde{X}_{\ol{\eta'}}$, so the natural
map $\pi_{1}\left(\widetilde{X}_{\ol{\eta'}}-\ol{\eta'}\right)\to\pi_{1}\left(\widetilde{X}_{\ol{\eta'}}\right)$
is surjective, and injectivity follows from the cited statement of
purity. Consequently, $g_{*}L$ is lisse in a neighborhood around
$\ol{\eta'}$, which is a contradiction.

Hence, $g_{*}L$ is lisse. We may successively apply the same argument
to the remaining divisors $D_{2},\ldots,D_{s}$ to conclude that $j_{*}L$
is lisse (because the local monodromy of $g_{*}L$ around each $D_{i}$
is also trivial).
\end{proof}
\end{prop}

\subsection{Elementary combinatorial lemmas on trees}

Recall that a rooted tree is simply a tree (an undirected graph that
is connected and has no cycles) in which one vertex is designated
to be the root. In a rooted tree, a parent of a vertex $v$ is the
unique vertex adjacent to $v$ that is on the path to the root; reciprocally,
a child of a vertex $v$ is a vertex for which $v$ is its unique
parent. The root, in particular, has no parent. A leaf is a vertex
with no children. A rooted planar tree is a rooted tree with an embedding
in the plane with the root at the top and the children of each vertex
$v$ lower than $v$.

Fix $r\ge1$ and non-negative integers $d_{1},\ldots,d_{m}$ such
that $d_{1}+\cdots+d_{m}=r$. Let $S_{r}$ denote the symmetric group
on a set of size $r$ and $S_{d_{1}}\times\cdots\times S_{d_{m}}$
be a Young subgroup.

Let $\mathcal{T}$ denote the set of rooted trees with exactly $r$
leaves, up to $S_{d_{1}}\times\cdots\times S_{d_{m}}$-action, such
that each non-leaf vertex aside from the root has valence at least
three. The leaves are labeled 1 to $r$, and the action is given by
$S_{d_{1}}$ acting on the leaves $1$ through $d_{1}$, $S_{d_{2}}$
acting on the leaves labeled $d_{1}+1$ through $d_{2}+d_{1}$, and
so on. 

For such a tree $T$ and a non-leaf vertex $v$, consider the $d(v)$
(i.e. the valency) neighbors of $v$, including the leaves. If we
remove the edge connecting $v$ to a neighbor $w$, then the connected
component containing $w$ can be viewed as a rooted tree $T_{w}$
with some number of labels. Each such neighbor defines a new rooted
sub-tree. Let $\lambda(v)$ be a partition of $d(v)$ and $S_{\lambda(v)}$
be the product of symmetric groups $S_{\lambda(v)_{1}}\times S_{\lambda(v)_{2}}\times\cdots$
defined as follows:

Consider the trees $T_{w}$ up to isomorphism (for a given $v$),
remembering which tree contains the original root (if $v\ne\text{root}$),
under the same action. Then, the partition $\lambda(v)$ is a collection
of numbers, where each number is the number of $w$ adjacent to $v$
such that $T_{w}$ lies in a particular isomorphism class.
\begin{example}
\label{exa:graphexample}The number of leaves adjacent to $v$ appears
in the partition, and if $v\ne\text{root}$, one of the summands of
the partition $\lambda(v)$ is simply 1, corresponding to the tree
$T_{w}$ which contains the original root. 
\end{example}

\begin{lem}
\label{lem:graph_theory-1}Suppose $T\in\mathcal{T}$. 
\begin{enumerate}
\item There are at most $r$ non-leaf vertices. 
\item The sum of the valences over all non-leaf vertices is at most $3r-2$. 
\item The sum of the valences over all non-leaf vertices minus the number
of non-leaf, non-root vertices is at most $2r-1$. 
\item The valency of any vertex is at most $r+1$.
\end{enumerate}
\begin{proof}
Suppose there are $s$ non-leaf, non-root vertices, whose valencies
are $d\left(v_{1}\right),\ldots,d\left(v_{s}\right)$. Then, since
the sum of all valencies (including those of the leaves) is twice
the number of edges and because a tree has one fewer edge than the
total number of vertices, we obtain 
\[
d\left(v_{1}\right)+\cdots+d\left(v_{s}\right)+d\left(\text{root}\right)+\underset{\text{contribution from the leaves}}{\underbrace{r}}=2\left(r+s+1-1\right),
\]
which means 
\[
d\left(v_{1}\right)+\cdots+d\left(v_{s}\right)+d\left(\text{root}\right)=r+2s.
\]
By assumption, the left-hand side is at least $3s+1$, which means
$s\le r-1$, i.e. there are at most $r$ non-leaf vertices. This also
implies $d\left(v_{1}\right)+\cdots+d\left(v_{s}\right)+d\left(\text{root}\right)\le3r-2$
and $d\left(v_{1}\right)+\cdots+d\left(v_{s}\right)+d\left(\text{root}\right)-s=r+s\le2r-1$.

Finally, if a vertex has more than $r+1$ neighbors, then it has more
than $r$ children, which means that $T$ has more than $r$ leaves.
\end{proof}
\end{lem}

\begin{lem}
\label{lem:planartrees-1}Let $T\in\mathcal{T}$. There are ${d(\text{root}) \choose \lambda(\text{root})}\prod_{v\ne\text{root}}{d(v)-1 \choose \lambda(v)\backslash\{1\}}$
distinct rooted planar trees isomorphic to $T$. 
\begin{proof}
We prove this inductively. The base case of one vertex is trivial.
Starting at the root, say with neighbors $w$, there are ${d(\text{root}) \choose \lambda(\text{root})}$
ways to rearrange the different $T_{w}$ (recall that $T$ and $T_{w}$
remember the information about the labeling of leaves up to $S_{d_{1}}\times\cdots\times S_{d_{m}}$-action).
Then, recursively, within each $T_{w}$, by induction, there are ${d(w) \choose \lambda(w)}\prod_{v}{d(v)-1 \choose \lambda(v)\backslash\{1\}}$
distinct rooted planar trees isomorphic to $T_{w}$, where $v$ runs
over all children of $w$ (in particular, $T_{w}$ does not contain
the original root). Taking the product over all $w$, the result follows.
\end{proof}
\end{lem}

\begin{lem}
\label{lem:numberplanartree-1}There are $O_{m}\left(\left(16m\right)^{r}\right)$
distinct rooted planar trees isomorphic to some tree in $\mathcal{T}$. 
\begin{proof}
Let $\mathcal{T}'$ be the same definition as $\mathcal{T}$, except
with $S_{r}$-action replacing $S_{d_{1}}\times\cdots\times S_{d_{m}}$-action.
Then, note that the number of trees in $\mathcal{T}$ is at most ${r \choose d_{1},\ldots,d_{m}}$
times the number of trees in $\mathcal{T}'$. 

We can think of any tree in $\mathcal{T}'$ as a rooted unlabelled
tree with at least $r+1$ vertices and at most $r+r=2r$ vertices
by Lemma \ref{lem:graph_theory-1}. 

By \cite{Stanley}, the number of rooted planar unlabelled trees with
$r'+1$ vertices is simply the $r'$-th Catalan number $C_{r'}=\frac{{2r' \choose r'}}{r'+1}$. 

Then, the number of distinct rooted planar trees isomorphic to a tree
in $\mathcal{T}$ is at most 
\[
{r \choose d_{1},\ldots,d_{m}}\left(C_{r}+\cdots+C_{2r-1}\right)\le\frac{r!}{\left(\left(r/m\right)!\right)^{m}}4^{2r-1}.
\]
By a variant of Stirling's approximation, we have 
\begin{align*}
{r \choose d_{1},\ldots,d_{m}}\left(C_{r}+\cdots+C_{2r}\right) & \ll_{m}\frac{\sqrt{2\pi r}\left(\frac{r}{e}\right)^{r}e^{\frac{1}{12r}}}{\left(\sqrt{2\pi\frac{r}{m}}\left(\frac{r/m}{e}\right)^{r/m}e^{\frac{1}{12\frac{r}{m}+1}}\right)^{m}}16^{r}\\
 & \ll_{m}\frac{m^{m/2}e^{\frac{1}{12r}-\frac{m^{2}}{12r+m}}}{\left(\sqrt{2\pi r}\right)^{m-1}}\left(16m\right)^{r}\\
 & \ll_{m}\left(16m\right)^{r}.
\end{align*}
\end{proof}
\end{lem}

\subsection{\label{subsec:generalbound}Bounds for the cohomology of lisse sheaves
on the moduli space of stable maps}

Fix a non-negative integer $r.$ Let us quickly recall the definition
of the Kontsevich moduli stack $\ol{\mathcal{M}}_{0,r}\left(\P^{1},1\right)$
and some related objects.

In this subsection, we assume everything is over $\C.$ 

An $r$-pointed pre-stable curve $\left(C,p_{1},\ldots,p_{r}\right)$
of arithmetic genus zero is a reduced connected projective variety
$C$ of dimension one and genus zero with at worst nodal singularities
and smooth points $p_{1},\ldots,p_{r}\in C$. A point on a pre-stable
curve is special if it is a node or one of the $p_{i}$'s. A stable
curve $\left(C,p_{1},\ldots,p_{r}\right)$ of genus zero is a pre-stable
curve with at least three special points. A stable map $\left(C,p_{1},\ldots,p_{r}\right)\to\P^{1}$
of degree one is a morphism from an $r$-pointed pre-stable curve
such that exactly one of the irreducible components of $C$ maps isomorphically
onto $\P^{1}$, the rest of the components are contracted, and each
contracted component has at least three special points. 

Then, the following are Deligne--Mumford moduli stacks:
\begin{enumerate}
\item $\mathcal{M}_{0,r}$ parameterizes $r$-pointed smooth stable curves
of genus zero. This space has dimension $r-3$.
\item $\mathcal{\ol M}_{0,r}$ parameterizes $r$-pointed stable curves
of genus zero. This space has dimension $r-3$.
\item $\mathcal{M}_{0,r}\left(\P^{1},1\right)$ parameterizes stable maps
$\left(C,p_{1},\ldots,p_{r}\right)\to\P^{1}$ of degree one with $C$
smooth. This space has dimension $r$.
\item $\ol{\mathcal{M}}_{0,r}\left(\P^{1},1\right)$ parameterizes stable
maps $\left(C,p_{1},\ldots,p_{r}\right)\to\P^{1}$ of degree one.
This space has dimension $r$.
\end{enumerate}
There is a canonical evaluation map $\ev\colon\mathcal{\ol M}_{0,r}\left(\P^{1},1\right)\to\left(\P^{1}\right)^{r}$
(c.f. \cite{Kock_Vainsencher}) that sends a stable map to the images
of its marked points. Let 
\[
\ol{\mathcal{M}}_{0,r}\left(\P^{1},1\right)_{\A^{r}}
\]
denote the preimage of $\A^{r}$ under $\ev$. 

Let $\left[-/S_{d_{1}}\times\cdots\times S_{d_{m}}\right]$ denote
the stack quotient. We freely use the notation of the previous subsection
on trees. For a tree $T\in\mathcal{T}$, let the stabilizer of $T$
under the action of $S_{d_{1}}\times\cdots\times S_{d_{m}}$ be denoted
by $\stab(T)$. 

Note that genus zero pre-stable and stable curves should be thought
of as collections of $\P^{1}$'s glued together at nodes without loops---this
leads to a stratification related to trees. 

\begin{lem}
\label{lem:stratification}$\left[\ol{\mathcal{M}}_{0,r}\left(\P^{1},1\right)_{\A^{r}}/S_{d_{1}}\times\cdots\times S_{d_{m}}\right]$
exhibits a stratification parametrized by trees in $\mathcal{T}$,
i.e. of the form 
\[
\mathcal{M}_{T}\coloneqq\left[\left(\left(\prod_{v\ne\text{root}}\mathcal{M}_{0,d(v)}\right)\times\mathcal{\mathcal{M}}_{0,d(\text{root})}\left(\P^{1},1\right)_{\A^{d(\text{root})}}\right)/\stab(T)\right]
\]
with $T\in\mathcal{T}$ and $v$ runs over all non-leaf and non-root
vertices.
\begin{proof}
Recall that the space $\ol{\mathcal{M}}_{0,r}\left(\P^{1},1\right)$
admits a stratification corresponding to rooted trees with exactly
$r$ leaves such that each non-leaf vertex aside from the root has
valence at least three, i.e. of the form 
\[
\left(\prod_{v\ne\text{root}}\mathcal{M}_{0,d(v)}\right)\times\mathcal{\mathcal{M}}_{0,d(\text{root})}\left(\P^{1},1\right).
\]
Then, $\ol{\mathcal{M}}_{0,r}\left(\P^{1},1\right)_{\A^{r}}$ admits
a stratification of the form 
\[
\left(\prod_{v\ne\text{root}}\mathcal{M}_{0,d(v)}\right)\times\mathcal{\mathcal{M}}_{0,d(\text{root})}\left(\P^{1},1\right)_{\A^{d(\text{root})}},
\]
from which the result follows.
\end{proof}
\end{lem}

\begin{lem}
\label{lem:fibration}Let $f\colon X\to Y$ be a morphism of smooth
Deligne--Mumford stacks over a characteristic zero field that is
etale-locally on $Y$ isomorphic to a projection map $F\times Y'\to Y'$.
Then, if $\mathcal{L}$ is a lisse sheaf on $X$, the derived pushforward
of $\mathcal{L}$, up to shift, is a lisse sheaf. 
\begin{proof}
Lisse-ness of $Rf_{*}\mathcal{L}$ is etale-local on $Y$, and since
$f$ is etale-locally a projection of the form $F\times Y\to Y$,
we may assume $X=F\times Y$. Moreover, in characteristic zero, since
$\pi_{1}\left(X\times Y\right)\cong\pi_{1}\left(X\right)\times\pi_{1}\left(Y\right)$,
any lisse sheaf on $F\times Y$ can be expressed as the box product
of a lisse sheaf on $F$ and a lisse sheaf on $Y$. Then, working
etale-locally on $Y$ again, we may assume $\mathcal{L}$ is the pull
back of a lisse sheaf from $F$. 

By proper base change (c.f. Tag 095S of \cite{stacks}), $Rf_{!}\mathcal{L}$
is isomorphic to the pushforward of a lisse sheaf on $F$ to a point,
then pull-backed to $Y$. In other words, $Rf_{!}\mathcal{L}$ is
lisse. 

If $D$ denotes the Verdier duality functor, then $Rf_{*}\mathcal{L}\cong D\left(Rf_{!}\left(D\left(\mathcal{L}\right)\right)\right)$.
Since $X$ is smooth, up to a shift, $D\left(\mathcal{L}\right)$
is lisse. Similarly, since $Y$ is smooth and $Rf_{!}\left(D\left(\mathcal{L}\right)\right)$
is lisse by above, it follows that $D\left(Rf_{!}\left(D\left(\mathcal{L}\right)\right)\right)$
is lisse, up to a shift.
\end{proof}
\end{lem}

By \cite{Kontsevich}, the complement of $\mathcal{M}_{0,r}\left(\P^{1},1\right)$
in $\mathcal{\ol M}_{0,r}\left(\P^{1},1\right)_{\A^{r}}$ is a normal
crossings divisor, more precisely a union of divisors $D_{1},\ldots,D_{s}$
comprising maps whose source is two $\P^{1}$'s (with one distinguished
$\P^{1}$ mapping isomorphically onto $\P^{1}$ with degree one).
\begin{prop}
\label{prop:cohomologybound} Let $S_{d_{1}}\times\cdots\times S_{d_{m}}$
be a Young subgroup of the symmetric group $S_{r}$. Suppose $U$
is an open substack of $\ol{\mathcal{M}}_{0,r}\left(\P^{1},1\right)_{\A^{r}}$
given by the complement of the union of a (possibly empty) subset
of the divisors $D_{i}$. Then, for any local system $\mathcal{L}$
of rank one on $\left[U/S_{d_{1}}\times\cdots\times S_{d_{m}}\right]$,
we have 
\[
\dim H^{*}\left(\left[U/S_{d_{1}}\times\cdots\times S_{d_{m}}\right],\mathcal{L}\right)\ll_{m}\left(64m\right)^{r}.
\]
\begin{proof}
For this proof, when we write $\dim H^{*}(-)$, we mean $\max_{i}\left\{ \dim H^{i}\left(-\right)\right\} ,$
which is well-defined for all situations in consideration.

By Lemma \ref{lem:stratification}, $\left[\ol{\mathcal{M}}_{0,r}\left(\P^{1},1\right)_{\A^{r}}/S_{d_{1}}\times\cdots\times S_{d_{m}}\right]$
admits a stratification corresponding to trees in $\mathcal{T}$,
i.e. of the form 
\[
\mathcal{M}_{T}=\left[\left(\left(\prod_{v\ne\text{root}}\mathcal{M}_{0,d(v)}\right)\times\mathcal{\mathcal{M}}_{0,d(\text{root})}\left(\P^{1},1\right)_{\A^{d(\text{root})}}\right)/\stab(T)\right]
\]
with $T\in\mathcal{T}$. 

This means that $\left[U/S_{d_{1}}\times\cdots\times S_{d_{m}}\right]$
admits a stratification of the same form, except with $\mathcal{T}$
replaced with a subset of $\mathcal{T}$. 

In particular,
\begin{align*}
\dim H^{*}\left(\left[U/S_{d_{1}}\times\cdots\times S_{d_{m}}\right],\mathcal{L}\right) & \le\sum_{T\in\mathcal{T}}\dim H^{*}\left(\mathcal{M}_{T},\mathcal{L}\right).
\end{align*}
For $T\in\mathcal{T}$, note that $\stab(T)$ is a semi-direct product
of $S_{\lambda(\text{root})}$ acting on $\left\{ \stab(T)\backslash S_{\lambda(\text{root})}\right\} \cup\left\{ \Id\right\} $
(the latter being a normal subgroup); indeed, note that the partition
$\lambda(v)$ for $v\ne\text{root}$ contains 1 as explained in Example
\ref{exa:graphexample}. Moreover, $\left\{ \stab(T)\backslash S_{\lambda(\text{root})}\right\} \cup\left\{ \Id\right\} $
is clearly the product of stabilizers of the trees $T_{v}$ for $v$
ranging over non-leaf, non-root vertices. Each such tree is also a
rooted tree, so iterating this process implies $\stab(T)$ is an iterated
semi-direct product of $S_{\lambda(v)}$'s. 

The Leray spectral sequence (c.f. Tag 03QA of \cite{stacks}) applied
to the quotient 
\[
q\colon\mathcal{M}_{T}\to\left[\mathcal{\mathcal{M}}_{0,d(\text{root})}\left(\P^{1},1\right)_{\A^{d(\text{root})}}/S_{\lambda(\text{root})}\right]
\]
gives the bound 
\[
\dim H^{*}\left(\mathcal{M}_{T},\mathcal{L}\right)\le\dim H^{*}\left(\left[\mathcal{\mathcal{M}}_{0,d(\text{root})}\left(\P^{1},1\right)_{\A^{d(\text{root})}}/S_{\lambda(\text{root})}\right],Rq_{*}\mathcal{L}\right).
\]

By Lemma \ref{lem:fibration}, $Rq_{*}\mathcal{L}$ is lisse and,
by proper base change (c.f. Tag 095S of \cite{stacks}), has rank
bounded by $\dim H^{*}\left(\left[\left(\prod_{v\ne\text{root}}\mathcal{M}_{0,d(v)}\right)/\left\{ \left\{ \stab(T)\backslash S_{\lambda(\text{root})}\right\} \cup\left\{ \Id\right\} \right\} \right],\mathcal{L}\right)$. 

Let $\PConf_{d}$ be the configuration space of $d$ points in $\C$.
The quotient $\PConf_{d}/S_{\lambda}$ with $\lambda$ a partition
of $d$ has a Fox--Neuwirth stratification by Euclidean spaces, whose
cells are determined by an ordered partition of $d$ of fixed length,
along with an assignment of colors: $\lambda_{1}$ points are colored
one color, $\lambda_{2}$ points are colored another, and so on. For
more details, c.f. \cite{Fox_Neuwirth}. 

Since there are ${d \choose \lambda}$ ways to assign colors and at
most $2^{d}$ ways to choose an ordered partition of $d$ of fixed
length, the number of Fox--Neuwirth cells of $\PConf_{d}\left(\C\right)/S_{\lambda}$
is given by ${d \choose \lambda}2^{d}$.

By definition, $\mathcal{\mathcal{M}}_{0,d(\text{root})}\left(\P^{1},1\right)_{\A^{d(\text{root})}}$
is the same as $\PConf_{d(\text{root})}$, so it follows that 
\[
\dim H^{*}\left(\mathcal{M}_{T},\mathcal{L}\right)\le{d(\text{root}) \choose \lambda(\text{root})}2^{d(\text{root})}\dim H^{*}\left(\left[\left(\prod_{v\ne\text{root}}\mathcal{M}_{0,d(v)}\right)/\left\{ \left\{ \stab(T)\backslash S_{\lambda(\text{root})}\right\} \cup\left\{ \Id\right\} \right\} \right],\mathcal{L}\right).
\]

Since $\stab(T)$ is an iterated semi-direct product of $S_{\lambda(v)}$'s,
we repeat this process using the Leray spectral sequence to obtain
\begin{align*}
 & \dim H^{*}\left(\left[\left(\prod_{v\ne\text{root}}\mathcal{M}_{0,d(v)}\right)/\left\{ \left\{ \stab(T)\backslash S_{\lambda(\text{root})}\right\} \cup\left\{ \Id\right\} \right\} \right],\mathcal{L}\right)\\
 & \le\dim H^{*}\left(\left[\left(\prod_{v\ne\text{root},w}\mathcal{M}_{0,d(v)}\right)/\left\{ \left\{ \stab(T)\backslash\left\{ S_{\lambda(\text{root})}\cup S_{\lambda(w)}\right\} \right\} \cup\left\{ \Id\right\} \right\} \right],R\left(q_{w}\right)_{*}\mathcal{L}\right)
\end{align*}
for some neighbor $w$ of the root (the idea being that $w$ should
be thought of as the root of $T_{w}$) and $q_{w}$ the projection
map to $\left(\prod_{v\ne\text{root},w}\mathcal{M}_{0,d(v)}\right)/\left\{ \left\{ \stab(T)\backslash\left\{ S_{\lambda(\text{root})}\cup S_{\lambda(w)}\right\} \right\} \cup\left\{ \Id\right\} \right\} $. 

Moreover, by the proof of Lemma 7.6 of \cite{EVW}, there is an $S_{b}$-equivariant
isomorphism 
\[
\PConf_{b}\cong\mathcal{M}_{0,b+1}\times\Aff,
\]
where $\Aff$ is the group scheme of upper triangular matrices corresponding
to $z\mapsto Az+B$. The map is given on $S$-points (for an arbitrary
scheme $S$) by sending collections of points $p_{1},\ldots,p_{b}\in\A^{1}\left(S\right)$
such that $p_{i}-p_{j}$ is a unit on $S$ for $i\ne j$ to $\left(\left(\P_{S}^{1};\infty,p_{1},p_{2},\ldots,p_{b}\right),\left(p_{2}-p_{1},p_{1}\right)\right)\in\mathcal{M}_{0,b+1}\left(S\right)\times\Aff(S)$.
Also, $\Aff\cong\G_{m}\times\G_{a}$. 

Let $v$ be a non-leaf and non-root vertex. By Example \ref{exa:graphexample},
every $\lambda(v)$ contains 1 (as an element of the partition), corresponding
to the root. Then, for any rank one local system $\mathcal{L}_{v}$
on $\left[\mathcal{M}_{0,d(v)}/S_{\lambda(v)}\right]$, we have 
\begin{align*}
\dim H^{*}\left(\left[\mathcal{M}_{0,d(v)}/S_{\lambda(v)}\right],\mathcal{L}_{v}\right) & =\dim H^{*}\left(\left[\mathcal{M}_{0,d(v)}/S_{\lambda(v)\backslash1}\right],\mathcal{L}_{v}\right)\\
 & \le\dim H^{*}\left(\left[\PConf_{d(v)-1}/S_{\lambda(v)\backslash1}\right],\mathcal{L}_{v}\boxtimes\Q_{\ell}\right),
\end{align*}
where in the first line $S_{\lambda(v)\backslash1}$ acts on the $d(v)-1$
points that do not correspond to the root, and in the second line
we use the Kunneth formula (Tag 0F13 of \cite{stacks}).

Then, by applying the same argument about counting Fox--Neuwirth
cells (and the lisse-ness of $R\left(q_{w}\right)_{*}\mathcal{L}$),
it follows that 
\begin{align*}
 & \dim H^{*}\left(\left[\left(\prod_{v\ne\text{root}}\mathcal{M}_{0,d(v)}\right)/\left\{ \left\{ \stab(T)\backslash S_{\lambda(\text{root})}\right\} \cup\left\{ \Id\right\} \right\} \right],\mathcal{L}\right)\\
 & \le{d(w)-1 \choose \lambda(w)\backslash1}2^{d(w)-1}\dim H^{*}\left(\left[\left(\prod_{v\ne\text{root},w}\mathcal{M}_{0,d(v)}\right)/\left\{ \left\{ \stab(T)\backslash\left\{ S_{\lambda(\text{root})}\cup S_{\lambda(w)}\right\} \right\} \cup\left\{ \Id\right\} \right\} \right],\mathcal{L}\right),
\end{align*}
which means 
\begin{align*}
\dim H^{*}\left(\mathcal{M}_{T},\mathcal{L}\right) & \le{d(w)-1 \choose \lambda(w)\backslash1}2^{d(w)-1}{d(\text{root}) \choose \lambda(\text{root})}2^{d(\text{root})}\\
 & \qquad\cdot\dim H^{*}\left(\left[\left(\prod_{v\ne\text{root},w}\mathcal{M}_{0,d(v)}\right)/\left\{ \left\{ \stab(T)\backslash\left\{ S_{\lambda(\text{root})}\cup S_{\lambda(w)}\right\} \right\} \cup\left\{ \Id\right\} \right\} \right],\mathcal{L}\right).
\end{align*}
Iterating this process, it follows that 
\begin{align*}
\dim H^{*}\left(\mathcal{M}_{T},\mathcal{L}\right) & \le\left(\prod_{v\ne\text{root}}{d(v)-1 \choose \lambda(v)\backslash1}2^{d(v)-1}\right){d(\text{root}) \choose \lambda(\text{root})}2^{d(\text{root})}.
\end{align*}

By Lemma \ref{lem:graph_theory-1}, the sum of valences of non-leaf
vertices minus the number of non-leaf, non-root vertices is at most
$2r-1$. As such, we have 
\begin{align*}
\dim H^{*}\left(\left[U/S_{d_{1}}\times\cdots\times S_{d_{m}}\right],\mathcal{L}\right) & \le\sum_{T\in\mathcal{T}}\dim H^{*}\left(\mathcal{M}_{T},\mathcal{L}\right)\\
 & \le\sum_{T\in\mathcal{T}}\left(\prod_{v\ne\text{root}}{d(v)-1 \choose \lambda(v)\backslash1}2^{d(v)-1}\right){d(\text{root}) \choose \lambda(\text{root})}2^{d(\text{root})}\\
 & \le2^{2r-1}\sum_{T\in\mathcal{T}}\left({d(\text{root}) \choose \lambda(\text{root})}\prod_{v\ne\text{root}}{d(v)-1 \choose \lambda(v)\backslash1}\right).
\end{align*}

By Lemma \ref{lem:planartrees-1}, $\sum_{T\in\mathcal{T}}\left({d(\text{root}) \choose \lambda(\text{root})}\prod_{v\ne\text{root}}{d(v)-1 \choose \lambda(v)\backslash1}\right)$
is simply the number of rooted planar trees isomorphic to some tree
in $\mathcal{T}$. By Lemma \ref{lem:numberplanartree-1}, this is
$O_{m}\left(\left(16m\right)^{r}\right)$. 

Then, we have 
\begin{align*}
\dim H^{*}\left(\left[U/S_{d_{1}}\times\cdots\times S_{d_{m}}\right],\mathcal{L}\right) & \ll_{m}4^{r}\cdot\left(16m\right)^{r}=\left(64m\right)^{r},
\end{align*}
as desired.

\end{proof}
\end{prop}

\section{Relationship between $a$-coefficients and Fourier transforms}

Let us continue to use the notation from Subsection \ref{subsec:notation}.
Moreover, as in the hypotheses of Theorem \ref{thm:mainresult}, we
assume for this section that 
\begin{enumerate}
\item $n$ is even and 
\item $M_{1,1}=\cdots=M_{1,s}=0$ and $M_{1,s+1},\ldots,M_{1,m}$ are non-zero,
with $s\ge1$. 
\end{enumerate}
For a prime polynomial $\pi\in\F_{q}[T]$, a non-trivial multiplicative
character $\chi\colon\left(\F_{q}[T]/\pi\right)^{\times}\to\C^{\times}$,
a non-trivial additive character $\psi\colon\F_{q}[T]/\pi\to\C^{\times}$,
and a polynomial $h\in\F_{q}[T]$ relatively prime to $\pi$, note
that 
\begin{align*}
\sum_{f\in\left(\F_{q}[T]/\pi\right)^{\times}}\chi(f)\psi\left(hf\right) & =\sum_{f\in\left(\F_{q}[T]/\pi\right)^{\times}}\chi\left(f/h\right)\psi\left(f\right)=\ol{\chi}\left(h\right)\sum_{f\in\left(\F_{q}[T]/\pi\right)^{\times}}\chi\left(f\right)\psi\left(f\right),
\end{align*}
which gives a relationship between the conjugate Dirichlet character
$\ol{\chi}$ and the Fourier transform of $\chi$. 

In this section, we establish a variant of this relationship for $a$-coefficients,
generalizing the strategy from Proposition 4.4 in \cite{sawin_general}. 
\begin{lem}
\label{lem:ncdonclosed}Whether a divisor $D$ on a variety $X$ is
a normal crossings divisors can be checked on closed points.
\begin{proof}
Since the strict Henselization of a local ring preserves (and reflects)
regularity, it suffices to show that if for every closed point $p\in D$,
$\mathcal{O}_{X,p}$ is regular and there is a regular system of parameters
$x_{1},\ldots,x_{d}\in\mathfrak{m}_{p}$ and $1\le r\le d$ such that
$x_{1}\cdots x_{r}$ cuts out $D$ at $p$, then $D$ is a strict
normal crossings divisor (c.f. Tag 0CBN in \cite{stacks}).

Note that in a variety, the set of closed points is dense. 

Let $p$ be an arbitrary point of $D$. By Serre's theorem on openness
of regularity, $\mathcal{O}_{X,p}$ is regular. Let $q$ be a closed
point that specializes $p$ (i.e. so that $\mathcal{O}_{X,p}$ is
a localization of $\mathcal{O}_{X,q}$). Pick a regular system of
parameters $x_{1},\ldots,x_{d}\in\mf m_{q}$ such that $D$ is cut
out by $x_{1}\cdots x_{r}$. Localization is flat, so this regular
sequence is mapped to a regular sequence in $\mathcal{O}_{X,p}.$ 
\end{proof}
\end{lem}

Recall that $\prod_{i=1}^{m}\A^{d_{i}}$ parametrizes tuples $\left(f_{1},\ldots,f_{m}\right)$
of monic polynomials. Let $X\subset\prod_{i=1}^{m}\A^{d_{i}}$ be
such that $\left(f_{1},f_{1}'\right)=\cdots=\left(f_{m},f_{m}'\right)=1$
and $\left(f_{i},f_{j}\right)=1$ for $(i,j)\not\in\left\{ (1,2),\ldots,(1,m)\right\} $,
i.e. square-free tuples $(f_{1},\ldots,f_{m})$ that are pairwise
coprime aside from $f_{1}$. 
\begin{lem}
\label{lem:complementisNCD}Fix non-negative integers $d_{1},\ldots,d_{m}$
such that $d_{1}\ge d_{s+1}+\cdots+d_{m}$. Consider the open subset
 $U\subset X$ such that $\left(f_{1},f_{i}\right)=1$ for all $i\ge s+1$
and $D$ the complement, i.e. $D$ is the locus of points such that
$f_{1}$ and some $f_{i}$ share a factor. 

Then, $D$ is a normal crossings divisor. 
\begin{proof}
We use Lemma \ref{lem:ncdonclosed}. 

Verifying that $D$ is a normal crossings divisor is an etale-local
statement, so we may pull back by the factorization/multiplication
maps $\underbrace{\A^{1}\times\A^{1}\times\cdots\times\A^{1}}_{d_{s+1}}\to\A^{d_{s+1}},\ldots,\underbrace{\A^{1}\times\A^{1}\times\cdots\times\A^{1}}_{d_{m}}\to\A^{d_{m}},$
which is etale on $X$ by assumption (c.f. the proof of Lemma 3.1
in \cite{sawin_general}). 

Then, for any $f_{j}(T)$, $D$ etale-locally near $f_{j}(t)$ looks
like the union of hyperplanes cut out by $T-\alpha_{ij}$, where the
$\alpha_{ij}$'s are the roots of $f_{j}$. For any $S_{j}\subset\{1,2,\ldots,d_{j}\}$,
ranging over $s+1\le j\le m$, the intersection of the corresponding
hyperplanes is 
\[
\left\{ f\in\A^{d_{1}}:\prod_{j=s+1}^{m}\prod_{i\in S_{j}}\left(T-\alpha_{ij}\right)|f\right\} \times\A^{d_{2}}\times\cdots\times\A^{d_{s}},
\]
which has the expected codimension $\left|\bigcup_{j=s+1}^{m}S_{j}\right|$,
since $\left|\bigcup_{j=s+1}^{m}S_{j}\right|\le d_{s+1}+\cdots+d_{m}\le d_{1}$. 
\end{proof}
\end{lem}

\begin{lem}
\label{lem:extenddefinitionofa}Using the notation of Lemma \ref{lem:complementisNCD},
let $\left(f_{1},f_{1}'\right)=\cdots=\left(f_{m},f_{m}'\right)=1$
and $\left(f_{i},f_{j}\right)=1$ for $(i,j)\not\in\left\{ (1,2),\ldots,(1,m)\right\} $.
Then, 
\[
a\left(f_{1},\ldots,f_{m};M\right)=\prod_{j>i\ge1}\left(\frac{f_{i}}{f_{j}}\right)_{\chi}^{M_{i,j}}\prod_{i\ge s+1}\left(\frac{f_{i}'}{f_{i}}\right)_{\chi}^{M_{i,i}}.
\]
 
\begin{proof}
The idea is that this formula is known on $U$ and we extend it to
$X$. More specifically, the $a$-coefficient vanishes where $f_{i}$
is not square-free for some $i\ge s+1$. 

By definition, $a\left(f_{1},\ldots,f_{m};M\right)$ is the trace
of Frobenius of the complex 
\[
K_{d_{1},\ldots,d_{m}}=j_{!*}\left(\mathcal{L}_{\chi}\left(F_{d_{1},\ldots,d_{m}}\right)\left[d_{1}+\cdots+d_{m}\right]\right)\left[-d_{1}-\cdots-d_{m}\right].
\]
Let $j'\colon U\subset X$ and $j''\colon X\subset\A^{d_{1}}\times\cdots\times\A^{d_{m}}$
as in Lemma \ref{lem:complementisNCD}.

Then, using Lemma \ref{lem:complementisNCD} to verify that $D$ is
a normal crossings divisor and observing that $\mathcal{L}_{\chi}\left(F_{d_{1},\ldots,d_{m}}\right)$
is lisse (by definition) and tame (by tameness of Kummer sheaves),
Proposition \ref{prop:IC_is_sheaf} then tells us that 
\[
K_{d_{1},\ldots,d_{m}}=j''_{!*}j'_{!}\left(\mathcal{L}_{\chi}\left(F_{d_{1},\ldots,d_{m}}\right)\left[d_{1}+\cdots+d_{m}\right]\right)\left[-d_{1}-\cdots-d_{m}\right].
\]
By taking the trace of Frobenius of this complex, noting that the
stalk at a point outside $U$ of the extension by zero is zero, the
result follows.
\end{proof}
\end{lem}

\begin{lem}
\label{lem:extraassumptionbeforedensity}Suppose $\deg f_{1}\ge\deg f_{s+1}\cdots f_{m}$,
i.e. $d_{1}\ge d_{s+1}+\cdots+d_{m}$. Also, assume $\left(f_{i},f_{i}'\right)=1$
for $i\ge1$ and $\left(f_{i},f_{j}\right)=1$ for $(i,j)\not\in\left\{ (1,2),\ldots,(1,m)\right\} $.
Then,
\[
a\left(f_{1},\ldots,f_{m};M'\right)=\fudge\left(d_{\ge s+1};M_{1,\ge s+1}\right)\sum_{h\in\mathcal{M}_{d_{s+1}+\cdots+d_{m}}}a\left(h,f_{2},\ldots,f_{m};M\right)e\left(\frac{hf_{1}}{f_{s+1}\cdots f_{m}}\right).
\]
\end{lem}

\begin{proof}
By Lemma \ref{lem:extenddefinitionofa}, we have 
\begin{align*}
 & \sum_{h\in\mathcal{M}_{\deg f_{s+1}\cdots f_{m}}}a\left(h,f_{2},\ldots,f_{m},M\right)e\left(\frac{hf_{1}}{f_{s+1}\cdots f_{m}}\right)\\
 & =\sum_{h\in\F_{q}[T]/f_{s+1}\cdots f_{m}}\prod_{i\ge s}\left(\frac{h}{f_{i}}\right)_{\chi}^{M_{1,i}}\prod_{j>i>1}\left(\frac{f_{i}}{f_{j}}\right)_{\chi}^{M_{i,j}}\prod_{i\ge1}\left(\frac{f_{i}'}{f_{i}}\right)_{\chi}^{M_{i,i}}e\left(\frac{hf_{1}}{f_{s+1}\cdots f_{m}}\right)\\
 & =\prod_{j>i>1}\left(\frac{f_{i}}{f_{j}}\right)_{\chi}^{M_{i,j}}\prod_{i\ge1}\left(\frac{f_{i}'}{f_{i}}\right)_{\chi}^{M_{i,i}}\sum_{h\in\F_{q}[T]/f_{s+1}\cdots f_{m}}\prod_{i\ge s+1}\left(\frac{h}{f_{i}}\right)_{\chi}^{M_{1,i}}e\left(\frac{hf_{1}}{f_{s+1}\cdots f_{m}}\right)\\
 & =\prod_{j>i>1}\left(\frac{f_{i}}{f_{j}}\right)_{\chi}^{M_{i,j}}\prod_{i\ge1}\left(\frac{f_{i}'}{f_{i}}\right)_{\chi}^{M_{i,i}}\sum_{h\in\F_{q}[T]/f_{s+1}\cdots f_{m}}\prod_{i\ge s+1}\left(\frac{hf_{1}^{-1}}{f_{i}}\right)_{\chi}^{M_{1,i}}e\left(\frac{h}{f_{s+1}\cdots f_{m}}\right)\\
 & =\prod_{j>i>1}\left(\frac{f_{i}}{f_{j}}\right)_{\chi}^{M_{i,j}}\prod_{i\ge1}\left(\frac{f_{i}'}{f_{i}}\right)_{\chi}^{M_{i,i}}\prod_{i\ge s}\left(\frac{f_{1}}{f_{i}}\right)_{\chi}^{-M_{1,i}}\sum_{h\in\F_{q}[T]/f_{s+1}\cdots f_{m}}\prod_{i\ge s+1}\left(\frac{h}{f_{i}}\right)_{\chi}^{M_{1,i}}e\left(\frac{h}{f_{s+1}\cdots f_{m}}\right).
\end{align*}
Write $h=h_{s+1}f_{s+2}\cdots f_{m}+\cdots+h_{m}f_{s+1}\cdots f_{m-1}.$
Then, 
\begin{align*}
 & \sum_{h\in\F_{q}[T]/f_{s+1}\cdots f_{m}}\prod_{i\ge s+1}\left(\frac{h}{f_{i}}\right)_{\chi}^{M_{1,i}}e\left(\frac{h}{f_{s+1}\cdots f_{m}}\right)\\
 & =\sum_{h\in\F_{q}[T]/f_{s+1}\cdots f_{m}}\prod_{i\ge s+1}\left(\frac{h_{s+1}f_{s+2}\cdots f_{m}+\cdots+h_{m}f_{s+1}\cdots f_{m-1}}{f_{i}}\right)_{\chi}^{M_{1,i}}e\left(\frac{h_{s+1}}{f_{s+1}}+\cdots+\frac{h_{m}}{f_{m}}\right)\\
 & =\left(\sum_{h_{s+1}\in\F_{q}[T]/f_{s+1}}\left(\frac{h_{s+1}f_{s+2}\cdots f_{m}}{f_{s+1}}\right)_{\chi}^{M_{1,s+1}}e\left(\frac{h_{s+1}}{f_{s+1}}\right)\right)\cdots\left(\sum_{h_{m}\in\F_{q}[T]/f_{m}}\left(\frac{h_{m}f_{s+1}\cdots f_{m-1}}{f_{m}}\right)_{\chi}^{M_{1,m}}e\left(\frac{h_{m}}{f_{m}}\right)\right)\\
 & =\prod_{i\ne j\ge s+1}\left(\frac{f_{j}}{f_{i}}\right)_{\chi}^{M_{1,i}}\prod_{i\ge s+1}G\left(\chi^{M_{1,i}},\psi\right)^{\deg f_{i}}(-1)^{\frac{\deg f_{i}(\deg f_{i}-1)(q-1)}{4}}\left(\frac{f_{i}'}{f_{i}}\right)_{\chi}^{M_{1,i}}\left(\frac{f_{i}'}{f_{i}}\right)_{\chi^{n/2}},
\end{align*}
where in the last line we use Lemma 2.4 of \cite{sawin_general} (and
the fact that $n$ is even). 

We have 
\begin{align*}
 & \sum_{h\in\mathcal{M}_{\deg f_{s+1}\cdots f_{m}}}a\left(h,f_{2},\ldots,f_{m},;M\right)e\left(\frac{hf_{1}}{f_{s+1}\cdots f_{m}}\right)\\
 & =\prod_{j>i>1}\left(\frac{f_{i}}{f_{j}}\right)_{\chi}^{M_{i,j}}\prod_{i=1}^{s}\left(\frac{f_{i}'}{f_{i}}\right)_{\chi}^{M_{i,i}}\prod_{i\ge s+1}\left(\frac{f_{i}'}{f_{i}}\right)_{\chi}^{M_{i,j}+M_{1,i}+\frac{n}{2}}\prod_{i\ge s+1}\left(\frac{f_{1}}{f_{i}}\right)_{\chi}^{-M_{1,i}}\prod_{i\ne j\ge s+1}\left(\frac{f_{j}}{f_{i}}\right)_{\chi}^{M_{1,i}}\\
 & \qquad\cdot\prod_{i\ge s+1}G\left(\chi^{M_{1,i}},\psi\right)^{\deg f_{i}}\prod_{i\ge s+1}(-1)^{\frac{\deg f_{i}(\deg f_{i}-1)(q-1)}{4}}\\
 & =\frac{\prod_{j>i>1,i\le s}\left(\frac{f_{i}}{f_{j}}\right)_{\chi}^{M_{i,j}}\prod_{j>i>s}\left(\frac{f_{i}}{f_{j}}\right)_{\chi}^{M_{i,j}+M_{1,i}+M_{1,j}}\prod_{i=1}^{s}\left(\frac{f_{i}'}{f_{i}}\right)_{\chi}^{M_{i,i}}\prod_{i>s}\left(\frac{f_{i}'}{f_{i}}\right)_{\chi}^{M_{i,i}+M_{1,i}+\frac{n}{2}}\prod_{i>s}\left(\frac{f_{1}}{f_{i}}\right)_{\chi}^{-M_{1,i}}}{\fudge\left(d_{\ge s+1};M_{1,\ge s+1}\right)},
\end{align*}
so the result follows.
\end{proof}
Next, we use the density trick in Lemma 4.6 of \cite{sawin_general}
to remove the assumption ``$\left(f_{i},f_{i}'\right)=1$ for $i\ge1$
and $\left(f_{i},f_{j}\right)=1$ for $(i,j)\not\in\left\{ (1,2),\ldots,(1,m)\right\} $''.
The general idea is to express both sides as trace functions of simple
perverse sheaves. Then, by Lemma \ref{lem:extraassumptionbeforedensity},
these trace functions agree on a dense open subset, which forces the
two perverse sheaves to in fact be the same.
\begin{lem}
\label{lem:relationshipfordegreebigger}Suppose $\deg f_{1}\ge\deg f_{s+1}\cdots f_{m}$,
i.e. $d_{1}\ge d_{s+1}+\cdots+d_{m}$. Then,
\[
a\left(f_{1},\ldots,f_{m};M'\right)=\fudge\left(d_{\ge s+1};M_{1,\ge s+1}\right)\sum_{h\in\mathcal{M}_{d_{s+1}+\cdots+d_{m}}}a\left(h,f_{2},\ldots,f_{m};M\right)e\left(\frac{hf_{1}}{f_{s+1}\cdots f_{m}}\right).
\]
\end{lem}

\begin{proof}
The argument is more-or-less identical to that of Lemma 4.6 of \cite{sawin_general},
so we will only sketch the argument. Let $d=d_{s+1}+\cdots+d_{m}$.
Recall the $\ell$-adic Fourier transform (\cite{Katz_Laumon}): Let
$p_{13}$ and $p_{23}$ be the two projections $\A^{d}\times\A^{d}\times\A^{d}\to\A^{d}\times\A^{d}$
and $\mu\colon\A^{d}\times\A^{d}\times\A^{d}\to\A^{1}$ be the dot
product of the first two factors. Then, the Fourier transform $\mathcal{F}_{\psi}\left(-\right)$
is defined as $p_{13!}\left(p_{23}^{*}\left(-\right)\otimes\mu^{*}\mathcal{L}_{\psi}\right)\left[d\right]$. 

Let $\sigma\colon\A^{d_{1}}\times\cdots\times\A^{d_{s}}\times\A^{d}\to\A^{d_{1}}\times\cdots\times\A^{d_{s}}\times\A^{d}$
be the morphism sending $\left(f_{1},\ldots,f_{m}\right)$ to 
\[
\left(t^{d}+\Res\left(\frac{t^{d-1}f_{1}}{f_{s+1}\cdots f_{m}}\right)t^{d_{1}-1}+\cdots+\Res\left(\frac{f_{1}}{f_{s+1}\cdots f_{m}}\right),f_{2},\ldots,f_{s},f_{s+1},\ldots,f_{m}\right)
\]
and $\alpha\colon\A^{d_{1}}\times\cdots\times\A^{d_{s}}\times\A^{d}\to\A^{1}$
is the morphism sending $\left(f_{1},\ldots,f_{m}\right)$ to $\Res\left(\frac{t^{d}f_{1}}{f_{s+1}\cdots f_{m}}\right)$. 

Then, the proof of Proposition 4.4 of \cite{sawin_general} shows
the trace function of 
\[
\sigma^{*}\mathcal{F}_{\psi}K_{d_{1},\ldots,d_{m};M}\otimes\alpha^{*}\mathcal{L}_{\psi},
\]
where the extra subscript in $K_{d_{1},\ldots,d_{m};M}$ denotes the
dependence on the matrix $M$, is
\[
(-1)^{d}\sum_{h\in\mathcal{M}_{d_{s+1}+\cdots+d_{m}}}a\left(h,f_{2},\ldots,f_{m};M\right)e\left(\frac{hf_{1}}{f_{s+1}\cdots f_{m}}\right).
\]
By the Hasse--Davenport relations, after multiplying by a factor
of $(-1)^{d}$, $\fudge\left(d_{\ge s+1};M_{1,\ge s+1}\right)$ is
a compatible system of Weil numbers, so there is a lisse rank one
sheaf $\mathcal{L}_{G}$ on $\Spec\F_{p}$ such that the trace function
of 
\[
\sigma^{*}\mathcal{F}_{\psi}K_{d_{1},\ldots,d_{m};M}\otimes\alpha^{*}\mathcal{L}_{\psi}\otimes\mathcal{L}_{G}
\]
is 
\[
\fudge\left(d_{\ge s+1};M_{1,\ge s+1}\right)\sum_{h\in\mathcal{M}_{d_{s+1}+\cdots+d_{m}}}a\left(h,f_{2},\ldots,f_{m};M\right)e\left(\frac{hf_{1}}{f_{s+1}\cdots f_{m}}\right).
\]
The proof of Proposition 4.4 of \cite{sawin_general} shows that the
sheaves $\sigma^{*}\mathcal{F}_{\psi}K_{d_{1},\ldots,d_{m};M}\otimes\alpha^{*}\mathcal{L}_{\psi}\otimes\mathcal{L}_{G}\left[d_{1}+\cdots+d_{m}\right]$
and $K_{d_{1},\ldots,d_{m};M'}\left[d_{1}+\cdots+d_{m}\right]$ are
simple perverse sheaves that have trace functions agreeing on the
dense open subset defined by $\left(f_{i},f_{i}'\right)=1$ for $i\ge1$
and $\left(f_{i},f_{j}\right)=1$ for $(i,j)\not\in\left\{ (1,2),\ldots,(1,m)\right\} $,
which, also by the proof of Proposition 4.4 of \cite{sawin_general},
implies they are isomorphic (in fact, they are both intermediate extensions
of the same lisse sheaf). The result follows.
\end{proof}
\begin{lem}
\label{lem:indepfordegreebigger}Suppose $\deg f,\deg g\ge\deg f_{s+1}\cdots f_{m}$
such that $f\equiv g\bmod f_{s+1}\cdots f_{m}$. Then, 
\[
a\left(f,f_{2},\ldots,f_{m};M\right)=a\left(g,f_{2},\ldots,f_{m};M\right).
\]
\begin{proof}
This follows from Lemma \ref{lem:relationshipfordegreebigger}. Indeed,
we have 
\begin{align*}
a\left(f,\ldots,f_{m};M\right) & =\fudge\left(d_{\ge s+1};M'_{1,\ge s+1}\right)\sum_{h\in\mathcal{M}_{d_{s+1}+\cdots+d_{m}}}a\left(h,f_{2},\ldots,f_{m};M'\right)e\left(\frac{hf}{f_{s+1}\cdots f_{m}}\right)\\
 & =\fudge\left(d_{\ge s+1};M'_{1,\ge s+1}\right)\sum_{h\in\mathcal{M}_{d_{s+1}+\cdots+d_{m}}}a\left(h,f_{2},\ldots,f_{m};M'\right)e\left(\frac{hg'}{f_{s+1}\cdots f_{m}}\right)\\
 & =a\left(g,\ldots,f_{m};M\right).
\end{align*}
\end{proof}
\end{lem}

Finally, combining the above steps, we are able to prove the relationship
in full generality:
\begin{prop}
\label{prop:relationship}Let $d_{1},\ldots,d_{m}$ be non-negative
integers and $f_{i}\in\mathcal{M}_{d_{i}}$ for each $i$. Then, we
have 
\[
a\left(f_{1},\ldots,f_{m};M'\right)=\fudge\left(d_{\ge s+1};M_{1,\ge s+1}\right)\sum_{h\in\mathcal{M}_{d_{s+1}+\cdots+d_{m}}}a\left(h,f_{2},\ldots,f_{m};M\right)e\left(\frac{hf_{1}}{f_{s+1}\cdots f_{m}}\right).
\]
\begin{proof}
Choose $v$ relatively prime to $f_{1}$ such that $\deg v+\deg f_{1}\ge d_{s+1}+\cdots+d_{m}.$ 

For $h\in\mathcal{M}_{d_{s+1}+\cdots+d_{m}}$, by Lemma \ref{lem:indepfordegreebigger},
we have 
\[
a\left(hv^{-1},f_{2},\ldots,f_{m};M\right)=a\left(hv^{-1}\%f_{s+1}\cdots f_{m}+f_{s+1}\cdots f_{m},f_{2},\ldots,f_{m};M\right),
\]
where $f\%g$ means the remainder of $f$ after dividing by $g$ (so
that $\deg\left(f\%g\right)<\deg g$), since $hv^{-1}$ and $hv^{-1}\%f_{s+1}\cdots f_{m}+f_{s+1}\cdots f_{m}$
both have degree at least $d_{s+1}+\cdots+d_{m}$ and have the same
remainder modulo $f_{s+1}\cdots f_{m}$. 

Then,
\begin{align*}
 & \sum_{h\in\mathcal{M}_{d_{s+1}+\cdots+d_{m}}}a\left(hv^{-1},f_{2},\ldots,f_{m};M\right)e\left(\frac{hf_{1}}{f_{s+1}\cdots f_{m}}\right)\\
 & =\sum_{h\in\mathcal{M}_{d_{s+1}+\cdots+d_{m}}}a\left(hv^{-1}\%f_{s+1}\cdots f_{m}+f_{s+1}\cdots f_{m},f_{2},\ldots,f_{m};M\right)e\left(\frac{hf_{1}}{f_{s+1}\cdots f_{m}}\right).
\end{align*}
As $h$ runs over $\mathcal{M}_{d_{s+1}+\cdots+d_{m}}$, observe that
$hv^{-1}\%f_{s+1}\cdots f_{m}+f_{s+1}\cdots f_{m}$ also runs over
$\mathcal{M}_{d_{s+1}+\cdots+d_{m}}$ (because any two polynomials
of degree less than $f_{s+1}\cdots f_{m}$ that are also congruent
to each other modulo $f_{s+1}\cdots f_{m}$ must be the same). 

It follows that 
\[
\sum_{h\in\mathcal{M}_{d_{s+1}+\cdots+d_{m}}}a\left(hv^{-1},f_{2},\ldots,f_{m};M\right)e\left(\frac{hf_{1}}{f_{s+1}\cdots f_{m}}\right)=\sum_{h\in\mathcal{M}_{d_{s+1}+\cdots+d_{m}}}a\left(h,f_{2},\ldots,f_{m};M\right)e\left(\frac{hf_{1}v}{f_{s+1}\cdots f_{m}}\right).
\]
Combining this with Lemma \ref{lem:relationshipfordegreebigger},
we obtain 
\begin{align*}
a\left(f_{1}v,f_{2},\ldots,f_{m};M'\right) & =\fudge\left(d_{\ge s+1};M_{1,\ge s+1}\right)\sum_{h\in\mathcal{M}_{d_{s+1}+\cdots+d_{m}}}a\left(h,f_{2},\ldots,f_{m};M\right)e\left(\frac{hf_{1}v}{f_{s+1}\cdots f_{m}}\right)\\
 & =\fudge\left(d_{\ge s+1};M_{1,\ge s+1}\right)\sum_{h\in\mathcal{M}_{d_{s+1}+\cdots+d_{m}}}a\left(hv^{-1},f_{2},\ldots,f_{m};M\right)e\left(\frac{hf_{1}}{f_{s+1}\cdots f_{m}}\right).
\end{align*}
By applying twisted multiplicativity to $a\left(f_{1}v,f_{2},\ldots,f_{m};M'\right)$
and $a\left(hv^{-1},f_{2},\ldots,f_{m};M\right)$, we can write the
above as 
\begin{align*}
 & a\left(f_{1},\ldots,f_{m};M'\right)\left(\frac{v}{f_{s+1}}\right)_{\chi}^{-M_{1,s+1}}\cdots\left(\frac{v}{f_{m}}\right)_{\chi}^{-M_{1,m}}\\
 & =\fudge\left(d_{\ge s+1};M_{1,\ge s+1}\right)\sum_{h\in\mathcal{M}_{d_{s+1}+\cdots+d_{m}}}a\left(h,f_{2},\ldots,f_{m};M\right)\left(\frac{v^{-1}}{f_{s+1}}\right)_{\chi}^{M_{1,s+1}}\cdots\left(\frac{v^{-1}}{f_{m}}\right)_{\chi}^{M_{1,m}}e\left(\frac{hf_{1}}{f_{s+1}\cdots f_{m}}\right).
\end{align*}
Since $\left(\frac{v}{f_{s+1}}\right)_{\chi}^{-M_{1,s+1}}\cdots\left(\frac{v}{f_{m}}\right)_{\chi}^{-M_{1,m}}=\left(\frac{v^{-1}}{f_{s+1}}\right)_{\chi}^{M_{1,s+1}}\cdots\left(\frac{v^{-1}}{f_{m}}\right)_{\chi}^{M_{1,m}}\ne0$,
we may cancel them from both sides and the result follows.
\end{proof}
\end{prop}

\section{Derivation of functional equations}

We now have all the necessary tools to prove the first step of Theorem
\ref{thm:mainresult}, i.e. the functional equations as a formal equality
of power series. Let us continue using the notation from the previous
section. 

We divide up the analysis into two cases, depending on the triviality
of $\prod_{i=s+1}^{m}\left(\frac{-}{f_{i}}\right)_{\chi}^{M_{1,i}}$
on $\F_{q}^{\times}$ (much like in the single-variable situation).
These comprise the Subsections \ref{subsec:-is-trivial} and \ref{subsec:-is-not}.
We then combine these cases to obtain the functional equation in Subsection
\ref{subsec:Putting-everything-together}. Finally, as a quick verification,
in Subsection \ref{subsec:A-short-verification}, we show that our
functional equations match up with Whitehead's from his thesis (c.f.
\cite{whitehead_multiple}).

For fixed monic polynomials $f_{2},\ldots,f_{m}$ with $\deg f_{i}=d_{i}$,
let $d=d_{s+1}+\cdots+d_{m}$. Also, write 
\[
S_{t;f_{2},\ldots,f_{m};M}=\sum_{f\in\mathcal{M}_{t}}a\left(f,f_{2},\ldots,f_{m};M\right)
\]
and 
\[
P_{f_{2},\ldots,f_{m};M}\left(u_{1}\right)=\sum_{t\ge0}S_{t;f_{2},\ldots,f_{m};M}u_{1}^{t}.
\]

\begin{lem}
\label{lem:formulaforS}For $0\le t\le d$, we have 
\begin{align*}
S_{t;f_{2},\ldots,f_{m};M} & =\frac{q^{t-d}}{\fudge\left(d_{\ge s+1};M_{1,\ge s+1}\right)}a\left(f_{s}\cdots f_{m},f_{2},\ldots,f_{m};M'\right)\\
 & \qquad+\frac{q^{t-d}}{\fudge\left(d_{\ge s+1};M_{1,\ge s+1}\right)}\left[\sum_{k=0}^{d-t-2}\sum_{\lambda\in\F_{q}^{\times}}\prod_{i=s+1}^{m}\left(\frac{\lambda}{f_{i}}\right)_{\chi}^{M_{1,i}}\sum_{f\in\mathcal{M}_{k}}a\left(f,f_{2},\ldots,f_{m};M'\right)\right.\\
 & \qquad\left.+\sum_{\lambda\in\F_{q}^{\times}}\prod_{i=s+1}^{m}\left(\frac{\lambda}{f_{i}}\right)_{\chi}^{M_{1,i}}\psi(-\lambda)\sum_{f\in\mathcal{M}_{d-t-1}}a\left(f,f_{2},\ldots,f_{m};M'\right)\right].
\end{align*}
\begin{proof}
Observe that 
\begin{align*}
 & \sum_{f\in\mathcal{M}_{t}}a\left(f,f_{2},\ldots,f_{m};M\right)\\
 & =\sum_{f\in\mathcal{P}_{<t}}a\left(T^{t}+f,f_{2},\ldots,f_{m};M\right)\\
 & =q^{t-d}\sum_{f\in\F_{q}[T]/f_{s+1}\cdots f_{m}}a\left(T^{t}+f,f_{2},\ldots,f_{m};M\right)\sum_{h\in\mathcal{P}_{<d-t}}e\left(\frac{hf}{f_{s+1}\cdots f_{m}}\right)\\
 & =q^{t-d}\sum_{f\in\F_{q}[T]/f_{s+1}\cdots f_{m}}a\left(f,f_{2},\ldots,f_{m};M\right)\sum_{h\in\mathcal{P}_{<d-t}}e\left(\frac{hf-hT^{t}}{f_{s+1}\cdots f_{m}}\right)\\
 & =q^{t-d}\sum_{h\in\mathcal{P}_{<d-t}}e\left(\frac{-hT^{t}}{f_{s+1}\cdots f_{m}}\right)\sum_{f\in\F_{q}[T]/f_{s+1}\cdots f_{m}}a\left(f,f_{2},\ldots,f_{m};M\right)e\left(\frac{hf}{f_{s+1}\cdots f_{m}}\right).
\end{align*}
Then, using Proposition \ref{prop:relationship}, we have 
\begin{align*}
 & \frac{\fudge\left(d_{\ge s+1};M_{1,\ge s+1}\right)}{q^{t-d}}\sum_{f\in\mathcal{M}_{t}}a\left(f,f_{2},\ldots,f_{m};M\right)\\
 & =\sum_{h\in\mathcal{P}_{<d-t}}e\left(\frac{-hT^{t}}{f_{s+1}\cdots f_{m}}\right)a\left(h,f_{2},\ldots,f_{m};M'\right)\\
 & =a\left(f_{s+1}\cdots f_{m},f_{2},\ldots,f_{m};M'\right)+\sum_{k=0}^{d-t-1}\sum_{\lambda\in\F_{q}^{\times}}\sum_{f\in\mathcal{M}_{k}}e\left(\frac{-\lambda T^{t}f}{f_{s+1}\cdots f_{m}}\right)a\left(\lambda f,f_{2},\ldots,f_{m};M'\right)\\
 & =a\left(f_{s+1}\cdots f_{m},f_{2},\ldots,f_{m};M'\right)+\sum_{k=0}^{d-t-1}\sum_{\lambda\in\F_{q}^{\times}}\sum_{f\in\mathcal{M}_{k}}e\left(\frac{-\lambda T^{t}f}{f_{s+1}\cdots f_{m}}\right)a\left(f,f_{2},\ldots,f_{m};M'\right)\prod_{i=s+1}^{m}\left(\frac{\lambda}{f_{i}}\right)_{\chi}^{M_{1,i}}\\
 & =a\left(f_{s+1}\cdots f_{m},f_{2},\ldots,f_{m};M'\right)+\sum_{k=0}^{d-t-2}\sum_{\lambda\in\F_{q}^{\times}}\prod_{i=s+1}^{m}\left(\frac{\lambda}{f_{i}}\right)_{\chi}^{M_{1i}}\sum_{f\in\mathcal{M}_{k}}a\left(f,f_{2},\ldots,f_{m};M'\right)\\
 & \qquad+\sum_{\lambda\in\F_{q}^{\times}}\prod_{i=s+1}^{m}\left(\frac{\lambda}{f_{i}}\right)_{\chi}^{M_{1,i}}\psi(-\lambda)\sum_{f\in\mathcal{M}_{d-t-1}}a\left(f,f_{2},\ldots,f_{m};M'\right).
\end{align*}
\end{proof}
\end{lem}

For $t\ge d$, we have $S_{t+1;f_{2},\ldots,f_{m};M}=qS_{t;f_{2},\ldots,f_{m};M}$,
so 
\begin{align*}
S_{d;f_{2},\ldots,f_{m};M}u_{1}^{d}+S_{d+1;f_{2},\ldots,f_{m};M}u_{1}^{d+1}+\cdots & =S_{d;f_{2},\ldots,f_{m};M}u^{d}\left(1+qu_{1}+q^{2}u_{1}^{2}+\cdots\right)=\frac{S_{d;f_{2},\ldots,f_{m};M}u_{1}^{d}}{1-qu_{1}}.
\end{align*}
Hence, we have the following lemma.
\begin{lem}
\label{lem:expressionforP}
\begin{align*}
P_{f_{2},\ldots,f_{m};M}\left(u_{1}\right) & =\sum_{t=0}^{d-1}S_{t;f_{2},\ldots,f_{m};M}u_{1}^{t}+\frac{S_{d;f_{2},\ldots,f_{m};M}u_{1}^{d}}{1-qu_{1}}.
\end{align*}
\end{lem}

\begin{lem}
\label{lem:differenceofS}Let $t\ge0$. Then, 
\begin{align*}
 & qS_{t;f_{2},\ldots,f_{m};M}-S_{t+1;f_{2},\ldots,f_{m};M}\\
 & =\frac{q^{t+1-d}}{\fudge\left(d_{\ge s+1};M_{1,\ge s+1}\right)}\left[\sum_{\lambda\in\F_{q}^{\times}}\prod_{i=s+1}^{m}\left(\frac{\lambda}{f_{i}}\right)_{\chi}^{M_{1,i}}\sum_{f\in\mathcal{M}_{d-t-2}}a(\left(f,f_{2},\ldots,f_{m};M'\right)\right.\\
 & \qquad\left.+\sum_{\lambda\in\F_{q}^{\times}}\prod_{i=s+1}^{m}\left(\frac{\lambda}{f_{i}}\right)_{\chi}^{M_{1,i}}\psi(-\lambda)\left(\sum_{f\in\mathcal{M}_{d-t-1}}a\left(f,f_{2},\ldots,f_{m};M'\right)-\sum_{f\in\mathcal{M}_{d-t-2}}a\left(f,f_{2},\ldots,f_{m};M'\right)\right)\right].
\end{align*}
\begin{proof}
For $t\ge d$, both sides are zero by the observation above. For $0\le t\le d-1$,
this follows directly from Lemma \ref{lem:formulaforS}.
\end{proof}
\end{lem}

We now split our analysis into two cases, depending on whether or
not $\prod_{i=s+1}^{m}\left(\frac{-}{f_{i}}\right)_{\chi}^{M_{1,i}}$
is trivial on \textbf{$\F_{q}^{\times}$}. 

\subsection{\label{subsec:-is-trivial}$\prod_{i=s+1}^{m}\left(\frac{-}{f_{i}}\right)_{\chi}^{M_{1,i}}$
is trivial on $\protect\F_{q}^{\times}$ }

We have $\prod_{i=s+1}^{m}\left(\frac{\lambda}{f_{i}}\right)_{\chi}^{M_{1,i}}=1$
and $\sum_{\lambda\in\F_{q}^{\times}}\psi(-\lambda)=-1$.
\begin{lem}
\label{lem:pre-functional}For all $t$, we have 
\[
qS_{t;f_{2},\ldots,f_{m};M}-S_{t+1;f_{2},\ldots,f_{m};M}=\frac{q^{t+1-d}}{\fudge\left(d_{\ge s+1};M_{1,\ge s+1}\right)}\left[qS_{d-t-2;f_{2},\ldots,f_{m};M'}-S_{d-t-1;f_{2},\ldots,f_{m};M'}\right].
\]
\begin{proof}
If $d=0$, then both sides are 0, so we may assume $d\ge1$. 

For $t\ge0$, by Lemma \ref{lem:differenceofS}, we have 
\begin{align*}
 & qS_{t;f_{2},\ldots,f_{m};M}-S_{t+1;f_{2},\ldots,f_{m};M}\\
 & =\frac{q^{t+1-d}}{\fudge\left(d_{\ge s+1};M_{1,\ge s+1}\right)}\left[q\sum_{f\in\mathcal{M}_{d-t-2}}a\left(f,f_{2},\ldots,f_{m};M'\right)-\sum_{f\in\mathcal{M}_{d-t-1}}a\left(f,f_{2},\ldots,f_{m};M'\right)\right]\\
 & =\frac{q^{t+1-d}}{\fudge\left(d_{\ge s+1};M_{1,\ge s+1}\right)}\left[qS_{d-t-2;f_{2},\ldots,f_{m};M'}-S_{d-t-1;f_{2},\ldots,f_{m};M'}\right].
\end{align*}
For $t\le-1$, since $d-t-2\ge0$ and $\prod_{i=s+1}^{m}\left(\frac{-}{f_{i}}\right)_{\chi}^{M_{1,i}'}$
is also trivial on $\F_{q}^{\times}$, we obtain
\[
qS_{d-t-2;f_{2},\ldots,f_{m};M'}-S_{d-t-1;f_{2},\ldots,f_{m};M'}=\frac{q^{-t-1}}{\fudge\left(d_{\ge s+1};M'_{1,\ge s+1}\right)}\left[qS_{t;f_{2},\ldots,f_{m};M}-S_{t+1;f_{2},\ldots,f_{m};M}\right],
\]
which implies
\begin{align*}
qS_{t;f_{2},\ldots,f_{m};M}-S_{t+1;f_{2},\ldots,f_{m};M} & =\frac{q^{t+1-d}}{\fudge\left(d_{\ge s+1};M_{1,\ge s+1}\right)}\left[qS_{d-t-2;f_{2},\ldots,f_{m};M'}-S_{d-t-1;f_{2},\ldots,f_{m};M'}\right],
\end{align*}
as desired.
\end{proof}
\end{lem}

We then obtain the following functional equation.
\begin{prop}
\label{prop:goodlocalfunc}
\begin{align*}
P_{f_{2},\ldots,f_{m};M}\left(u_{1}\right)\left(qu_{1}-1\right) & =\frac{1}{\fudge\left(d_{\ge s+1};M_{1,\ge s+1}\right)}u_{1}^{d-1}\left(1-u_{1}\right)P_{f_{2},\ldots,f_{m};M'}\left(\frac{1}{qu_{1}}\right).
\end{align*}
\begin{proof}
Using Lemmas \ref{lem:expressionforP} and \ref{lem:pre-functional},
we have 
\begin{align*}
 & P_{f_{2},\ldots,f_{m};M}\left(u_{1}\right)\left(qu_{1}-1\right)\\
 & =\left(qu_{1}-1\right)\sum_{t=0}^{d-1}S_{t;f_{2},\ldots,f_{m};M}u_{1}^{t}-S_{d;f_{2},\ldots,f_{m};M}u_{1}^{d}\\
 & =\sum_{t=0}^{d-1}\left(qS_{t;f_{2},\ldots,f_{m};M}u_{1}^{t+1}-S_{t;f_{2},\ldots,f_{m};M}u_{1}^{t}\right)-S_{d;f_{2},\ldots,f_{m};M}u_{1}^{d}\\
 & =\left(qS_{-1;f_{2},\ldots,f_{m};M}-S_{0;f_{2},\ldots,f_{m};M}\right)+\cdots+\left(qS_{d-1;f_{2},\ldots,f_{m};M}u_{1}^{d}-S_{d;f_{2},\ldots,f_{m};M}u_{1}^{d}\right)\\
 & =\frac{\left(q^{1-d}S_{d-1;f_{2},\ldots,f_{m};M'}-q^{-d}S_{d;f_{2},\ldots,f_{m};M'}\right)+\cdots+\left(qS_{-1;f_{2},\ldots,f_{m};M'}-S_{0;f_{2},\ldots,f_{m};M'}\right)u_{1}^{d}}{\fudge\left(d_{\ge s+1};M_{1,\ge s+1}\right)}\\
 & =\frac{1}{\fudge\left(d_{\ge s+1};M_{1,\ge s+1}\right)}\left[\left(qS_{-1;f_{2},\ldots,f_{m};M'}u_{1}^{d}+\cdots+q^{2-d}S_{d-2;f_{2},\ldots,f_{m};M'}u_{1}+q^{1-d}S_{d-1;f_{2},\ldots,f_{m};M'}\right)\right.\\
 & \qquad\left.-\left(S_{0;f_{2},\ldots,f_{m};M'}u_{1}^{d}+\cdots+q^{1-d}S_{d-1;f_{2},\ldots,f_{m};M'}u_{1}+q^{-d}S_{d;f_{2},\ldots,f_{m};M'}\right)\right]\\
 & =\frac{1}{\fudge\left(d_{\ge s+1};M_{1,\ge s+1}\right)}\left[\left(u_{1}^{d-1}-u_{1}^{d}\right)P_{f_{2},\ldots,f_{m};M'}\left(\frac{1}{qu_{1}}\right)\right]\\
 & =\frac{1}{\fudge\left(d_{\ge s+1};M_{1,\ge s+1}\right)}u_{1}^{d-1}\left(1-u_{1}\right)P_{f_{2},\ldots,f_{m};M'}\left(\frac{1}{qu_{1}}\right),
\end{align*}
as desired.
\end{proof}
\end{prop}

\subsection{\label{subsec:-is-not}$\prod_{i=s+1}^{m}\left(\frac{-}{f_{i}}\right)_{\chi}^{M_{1,i}}$
is not trivial on $\protect\F_{q}^{\times}$ }

Note that $\prod_{i=s+1}^{m}\left(\frac{\lambda}{f_{i}}\right)_{\chi}^{M_{1,i}}=\prod_{i=s+1}^{m}\chi(\lambda)^{d_{i}M_{1,i}}=\chi(\lambda)^{\sum_{i=s+1}^{m}d_{i}M_{1,i}}$.
\begin{lem}
\label{lem:a_is_0}$a\left(f_{s+1}\cdots f_{m},f_{2},\ldots,f_{m};M'\right)=0$.
\begin{proof}
By Proposition \ref{prop:relationship}, we have 
\begin{align*}
a\left(f_{s+1}\cdots f_{m},f_{2},\ldots,f_{m};M'\right) & =\fudge\left(d_{\ge s+1};M_{1,\ge s+1}\right)\sum_{h\in\mathcal{M}_{d_{s+1}+\cdots+d_{m}}}a\left(h,f_{2},\ldots,f_{m};M\right)e\left(\frac{hf_{s+1}\cdots f_{m}}{f_{s+1}\cdots f_{m}}\right)\\
 & =\fudge\left(d_{\ge s+1};M_{1,\ge s+1}\right)\sum_{h\in\mathcal{M}_{d_{s+1}+\cdots+d_{m}}}a\left(h,f_{2},\ldots,f_{m};M\right).
\end{align*}
Let $S=\sum_{h\in\mathcal{M}_{d_{s+1}+\cdots+d_{m}}}a\left(h,f_{2},\ldots,f_{m};M\right)$
and $v\in\F_{q}^{\times}$ such that $\prod_{i=s+1}^{m}\left(\frac{v}{f_{i}}\right)_{\chi}^{M_{1,i}}\ne1$.
Then, 
\[
\sum_{h\in\mathcal{M}_{d_{s+1}+\cdots+d_{m}}}a\left(h,f_{2},\ldots,f_{m};M\right)=\sum_{h\in\mathcal{M}_{d_{s+1}+\cdots+d_{m}}}a\left(h\left(v+f_{s+1}\cdots f_{m}\right),f_{2},\ldots,f_{m};M\right),
\]
since we are summing over a full set of representatives. 

Then, by twisted multiplicativity, we have 
\begin{align*}
S & =\sum_{h\in\mathcal{M}_{d_{s+1}+\cdots+d_{m}}}a\left(h,f_{2},\ldots,f_{m};M\right)\\
 & =\sum_{h\in\mathcal{M}_{d_{s+1}+\cdots+d_{m}}}a\left(h\left(v+f_{s+1}\cdots f_{m}\right),f_{2},\ldots,f_{m};M\right)\\
 & =a\left(v+f_{s+1}\cdots f_{m},1,\ldots,1;M\right)\left(\frac{v+f_{s+1}\cdots f_{m}}{f_{s+1}}\right)_{\chi}^{M_{1,s+1}}\cdots\left(\frac{v+f_{s+1}\cdots f_{m}}{f_{m}}\right)_{\chi}^{M_{1,m}}S\\
 & =\prod_{i=s+1}^{m}\left(\frac{v}{f_{i}}\right)_{\chi}^{M_{1,i}}S.
\end{align*}
So $S=0$, which means $a\left(f_{s+1}\cdots f_{m},f_{2},\ldots,f_{m};M'\right)=0$. 
\end{proof}
\end{lem}

\begin{lem}
\label{lem:pre-functional2}For all $t$, we have 
\[
S_{t;f_{2},\ldots,f_{m};M}=\frac{q^{t-d}\chi(-1)^{\sum_{i=s+1}^{m}d_{i}M_{1,i}}G\left(\chi^{\sum_{i=s+1}^{m}d_{i}M_{1,i}},\psi\right)}{\fudge\left(d_{\ge s+1};M_{1,\ge s+1}\right)}S_{d-t-1;f_{2},\ldots,f_{m};M'}.
\]
\begin{proof}
By Proposition \ref{prop:relationship}, we have 
\begin{align*}
S_{d;f_{2},\ldots,f_{m};M} & =\sum_{f\in\F_{q}[T]/f_{s+1}\cdots f_{m}}a\left(f,f_{2},\ldots,f_{m};M\right)\\
 & =\frac{1}{\fudge\left(d_{\ge s+1};M_{1,\ge s+1}\right)}a\left(f_{s+1}\cdots f_{m},f_{2},\ldots,f_{m};M'\right),
\end{align*}
and by Lemma \ref{lem:a_is_0}, this expression is zero. Since $S_{t+1;f_{2},\ldots,f_{m};M}=qS_{t;f_{2},\ldots,f_{m};M}$
for $t\ge d$, it follows that the statement of the lemma holds for
$t\ge d$ and $t\le-1$. 

So assume $0\le t\le d-1$. By \ref{lem:formulaforS}, we have 
\begin{align*}
S_{t;f_{2},\ldots,f_{m};M} & =\frac{q^{t-d}}{\fudge\left(d_{\ge s+1};M_{1,\ge s+1}\right)}\left[\sum_{k=0}^{d-t-2}\sum_{\lambda\in\F_{q}^{\times}}\prod_{i=s+1}^{m}\left(\frac{\lambda}{f_{i}}\right)_{\chi}^{M_{1,i}}\sum_{f\in\mathcal{M}_{k}}a\left(f,f_{2},\ldots,f_{m};M'\right)\right.\\
 & \qquad\left.+\sum_{\lambda\in\F_{q}^{\times}}\prod_{i=s+1}^{m}\left(\frac{\lambda}{f_{i}}\right)_{\chi}^{M_{1,i}}\psi(-\lambda)\sum_{f\in\mathcal{M}_{d-t-1}}a\left(f,f_{2},\ldots,f_{m};M'\right)\right].\\
 & =\frac{q^{t-d}}{\fudge\left(d_{\ge s+1};M_{1,\ge s+1}\right)}\left[\sum_{\lambda\in\F_{q}^{\times}}\prod_{i=s+1}^{m}\left(\frac{\lambda}{f_{i}}\right)_{\chi}^{M_{1,i}}\psi(-\lambda)\sum_{f\in\mathcal{M}_{d-t-1}}a\left(f,f_{2},\ldots,f_{m};M'\right)\right]\\
 & =\frac{q^{t-d}}{\fudge\left(d_{\ge s+1};M_{1,\ge s+1}\right)}\left[\chi(-1)^{\sum_{i=s+1}^{m}d_{i}M_{1,i}}G\left(\chi^{\sum_{i=s+1}^{m}d_{i}M_{1,i}},\psi\right)S_{d-t-1;f_{2},\ldots,f_{m};M'}\right],
\end{align*}
as desired.
\end{proof}
\end{lem}

We then obtain the following functional equation.
\begin{prop}
\label{prop:badlocalfunc}
\[
P_{f_{2},\ldots,f_{m};M}(u_{1})=\frac{\chi(-1)^{\sum_{i=s+1}^{m}d_{i}M_{1,i}}G\left(\chi^{\sum_{i=s+1}^{m}d_{i}M_{1,i}},\psi\right)}{q\fudge\left(d_{\ge s+1};M_{1,\ge s+1}\right)}u_{1}^{d-1}P_{f_{2},\ldots,f_{m};M'}\left(\frac{1}{qu_{1}}\right).
\]
\begin{proof}
Indeed, we have 
\begin{align*}
u_{1}^{d-1}P_{f_{2},\ldots,f_{m};M'}\left(\frac{1}{qu_{1}}\right) & =u_{1}^{d-1}\left(\sum_{t=0}^{d-1}S_{t;f_{2},\ldots,f_{m};M'}q^{-t}u_{1}^{-t}+\frac{S_{d;f_{2},\ldots,f_{m};M'}q^{-d}u_{1}^{-d}}{1-u_{1}^{-1}}\right)\\
 & =\sum_{t=0}^{d-1}S_{t;f_{2},\ldots,f_{m};M'}q^{-t}u_{1}^{d-t-1}.
\end{align*}
So, by Lemma \ref{lem:pre-functional2} and changing variable ($t\mapsto d-t-1$),
we obtain
\begin{align*}
P_{f_{2},\ldots,f_{m};M}(u_{1}) & =\sum_{t=0}^{d-1}\frac{q^{t-d}\chi(-1)^{\sum_{i=s+1}^{m}d_{i}M_{1,i}}G\left(\chi^{\sum_{i=s+1}^{m}d_{i}M_{1,i}},\psi\right)}{\fudge\left(d_{\ge s+1};M_{1,\ge s+1}\right)}S_{d-t-1;f_{2},\ldots,f_{m};M'}u_{1}^{t}\\
 & =\frac{\chi(-1)^{\sum_{i=s+1}^{m}d_{i}M_{1,i}}G\left(\chi^{\sum_{i=s+1}^{m}d_{i}M_{1,i}},\psi\right)}{\fudge\left(d_{\ge s+1};M_{1,\ge s+1}\right)}\sum_{t=0}^{d-1}q^{-1}q^{-t}S_{t;f_{2},\ldots,f_{m};M'}u_{1}^{d-t-1}\\
 & =\frac{\chi(-1)^{\sum_{i=s+1}^{m}d_{i}M_{1i}}G\left(\chi^{\sum_{i=s+1}^{m}d_{i}M_{1,i}},\psi\right)}{q\fudge\left(d_{\ge s+1};M_{1,\ge s+1}\right)}u_{1}^{d-1}P_{f_{2},\ldots,f_{m};M'}\left(\frac{1}{qu_{1}}\right),
\end{align*}
as desired.
\end{proof}
\end{prop}

\subsection{\label{subsec:Putting-everything-together}Putting everything together}

We now complete the first step of the proof of Theorem \ref{thm:mainresult}.

Let $\zeta_{n}$ be a primitive $n$th root of unity, e.g. $\zeta_{n}=e^{2\pi i/n}$. 

Recall the multiple Dirichlet series
\begin{align*}
L\left(u_{1},\ldots,u_{m};M\right) & =\sum_{d_{2},\ldots,d_{m}\ge0}\sum_{f_{2}\in\mathcal{M}_{d_{2}},\ldots,f_{m}\in\mathcal{M}_{d_{m}}}\sum_{t\ge0}\sum_{f\in\mathcal{M}_{t}}a\left(f,f_{2},\ldots,f_{m};M\right)u_{1}^{t}u_{2}^{d_{2}}\cdots u_{m}^{d_{m}}\\
 & =\sum_{d_{2},\ldots,d_{m}\ge0}\sum_{f_{2}\in\mathcal{M}_{d_{2}},\ldots,f_{m}\in\mathcal{M}_{d_{m}}}P_{f_{2},\ldots,f_{m};M}\left(u_{1}\right)u_{2}^{d_{2}}\cdots u_{m}^{d_{m}}
\end{align*}
and 
\begin{align*}
 & L_{\fudge}\left(u_{1},\ldots,u_{m};M\right)\\
 & =\sum_{f_{1},\ldots,f_{s}\in\F_{q}[t]^{+}}\sum_{d_{s+1},\ldots,d_{m}}b\left(d_{\ge s+1};M_{1,\ge s+1}\right)\sum_{f_{s+1}\in\mathcal{M}_{d_{s+1}},\ldots,f_{m}\in\mathcal{M}_{d_{m}}}a\left(f_{1},\ldots,f_{m};M'\right)u_{1}^{d_{1}}\cdots u_{m}^{d_{m}}.
\end{align*}
Note that the ``roots-of-unity filter'' picks out only the terms
$u_{1}^{d_{1}}\cdots u_{m}^{d_{m}}$ such that $\sum_{i=s+1}^{m}d_{i}M_{1,i}$
is divisible by $n$:
\begin{align*}
 & \frac{1}{n}\sum_{0\le j\le n-1}L\left(u_{1},\ldots,u_{s},\zeta_{n}^{jM_{1,s+1}}u_{s+1},\ldots,\zeta_{n}^{jM_{1,m}}u_{m};M\right)\\
 & =\sum_{\underset{n|\sum_{i=s+1}^{m}d_{i}M_{1,i}}{d_{2},\ldots,d_{m}}}\sum_{f_{2}\in\mathcal{M}_{d_{2}},\ldots,f_{m}\in\mathcal{M}_{d_{m}}}P_{f_{2},\ldots,f_{m};M}\left(u_{1}\right)u_{2}^{d_{2}}\cdots u_{m}^{d_{m}}.
\end{align*}
So, by Proposition \ref{prop:badlocalfunc}, we have 
\begin{align*}
 & L\left(u_{1},\ldots,u_{m};M\right)-\frac{1}{n}\sum_{0\le j\le n-1}L\left(u_{1},\ldots,u_{s},\zeta_{n}^{jM_{1,s+1}}u_{s+1},\ldots,\zeta_{n}^{jM_{1,m}}u_{m};M\right)\\
 & =\sum_{\underset{n\nmid\sum_{i=s+1}^{m}d_{i}M_{1,i}}{d_{2},\ldots,d_{m}}}\sum_{f_{2}\in\mathcal{M}_{d_{2}},\ldots,f_{m}\in\mathcal{M}_{d_{m}}}P_{f_{2},\ldots,f_{m};M}\left(u_{1}\right)u_{2}^{d_{2}}\cdots u_{m}^{d_{m}}\\
 & =\sum_{\underset{n\nmid\sum_{i=s+1}^{m}d_{i}M_{1,i}}{d_{2},\ldots,d_{m}}}\sum_{f_{i}\in\mathcal{M}_{d_{i}}}\frac{\chi(-1)^{\sum_{i=s+1}^{m}d_{i}M_{1,i}}G\left(\chi^{\sum_{i=s+1}^{m}d_{i}M_{1,i}},\psi\right)}{q\fudge\left(d_{\ge s+1};M_{1,\ge s+1}\right)}u_{1}^{d-1}P_{f_{2},\ldots,f_{m};M'}\left(\frac{1}{qu_{1}}\right)u_{2}^{d_{2}}\cdots u_{m}^{d_{m}}\\
 & =\frac{1}{u_{1}}\sum_{\underset{n\nmid\sum_{i=s+1}^{m}d_{i}M_{1,i}}{d_{2},\ldots,d_{m}}}\sum_{f_{2}\in\mathcal{M}_{d_{2}},\ldots,f_{m}\in\mathcal{M}_{d_{m}}}\left(\frac{\chi(-1)^{\sum_{i=s+1}^{m}d_{i}M_{1,i}}G\left(\chi^{\sum_{i=s+1}^{m}d_{i}M_{1,i}},\psi\right)}{q^{1+\sum_{i=s+1}^{m}d_{i}/2}\fudge\left(d_{\ge s+1};M_{1,\ge s+1}\right)}P_{f_{2},\ldots,f_{m};M'}\left(\frac{1}{qu_{1}}\right)\right.\\
 & \qquad\left.\cdot\prod_{j=2}^{s}u_{j}^{d_{j}}\prod_{j=s+1}^{m}\left(q^{1/2}u_{1}u_{j}\right)^{d_{j}}\right)\\
 & =\frac{1}{u_{1}}L_{\fudge}\left(\frac{1}{qu_{1}},u_{2},\ldots u_{s},q^{1/2}u_{1}u_{s+1},\ldots,q^{1/2}u_{1}u_{m};M\right)\\
 & \qquad-\frac{1}{u_{1}}\frac{1}{n}\sum_{0\le j\le n-1}L_{\fudge}\left(\frac{1}{qu_{1}},u_{2},\ldots,u_{s},\zeta_{n}^{jM_{1,s+1}}q^{1/2}u_{1}u_{s+1},\ldots,\zeta_{n}^{jM_{1,m}}q^{1/2}u_{1}u_{m};M\right).
\end{align*}
Also, by Proposition \ref{prop:goodlocalfunc}, we have 
\begin{align*}
 & \frac{1}{n}\sum_{0\le j\le n-1}\left(qu_{1}-1\right)L\left(u_{1},\ldots,u_{s},\zeta_{n}^{jM_{1,s+1}}u_{s+1},\ldots,\zeta_{n}^{jM_{1,m}}u_{m};M\right)\\
 & =\sum_{\underset{n|\sum_{i=s+1}^{m}d_{i}M_{1,i}}{d_{2},\ldots,d_{m}}}\sum_{f_{2}\in\mathcal{M}_{d_{2}},\ldots,f_{m}\in\mathcal{M}_{d_{m}}}P_{f_{2},\ldots,f_{m};M}(u_{1})u_{2}^{d_{2}}\cdots u_{m}^{d_{m}}(qu_{1}-1)\\
 & =\sum_{\underset{n|\sum_{i=s+1}^{m}d_{i}M_{1,i}}{d_{2},\ldots,d_{m}}}\sum_{f_{i}\in\mathcal{M}_{d_{i}}}\frac{u_{1}^{d-1}\left(1-u_{1}\right)}{q^{\sum_{i=s+1}^{m}d_{i}/2}\fudge\left(d_{\ge s+1};M_{1,\ge s+1}\right)}P_{f_{2},\ldots,f_{m};M'}\left(\frac{1}{qu_{1}}\right)\prod_{j=2}^{s}u_{j}^{d_{j}}\prod_{j=s+1}^{m}\left(q^{1/2}u_{1}u_{j}\right)^{d_{j}}\\
 & =\frac{1}{n}\sum_{0\le j\le n-1}L_{\fudge}\left(\frac{1}{qu_{1}},u_{2},\ldots,u_{s},\zeta_{n}^{jM_{1,s+1}}q^{1/2}u_{1}u_{s+1},\ldots,\zeta_{n}^{jM_{1,m}}q^{1/2}u_{1}u_{m};M\right)\frac{1-u_{1}}{u_{1}}.
\end{align*}
Then, combining these two computations yields 
\begin{align*}
 & \frac{qu_{1}-1}{u_{1}}L_{\fudge}\left(\frac{1}{qu_{1}},u_{2},\ldots u_{s},q^{1/2}u_{1}u_{s+1},\ldots,q^{1/2}u_{1}u_{m};M\right)\\
 & \qquad-\frac{qu_{1}-1}{u_{1}}\frac{1}{n}\sum_{0\le j\le n-1}L_{\fudge}\left(\frac{1}{qu_{1}},u_{2},\ldots,u_{s},\zeta_{n}^{jM_{1,s+1}}q^{1/2}u_{1}u_{s+1},\ldots,\zeta_{n}^{jM_{1,m}}q^{1/2}u_{1}u_{m};M\right)\\
 & =\left(qu_{1}-1\right)L\left(u_{1},\ldots,u_{m};M\right)\\
 & \qquad+\left(1-\frac{1}{u_{1}}\right)\frac{1}{n}\sum_{0\le j\le n-1}L_{\fudge}\left(\frac{1}{qu_{1}},u_{2},\ldots,u_{s},\zeta_{n}^{jM_{1,s+1}}q^{1/2}u_{1}u_{s+1},\ldots,\zeta_{n}^{jM_{1,m}}q^{1/2}u_{1}u_{m};M\right).
\end{align*}
Re-arranging, we finally have 
\begin{align*}
 & u_{1}\left(qu_{1}-1\right)L\left(u_{1},\ldots,u_{m};M\right)\\
 & =\left(qu_{1}-1\right)L_{\fudge}\left(\frac{1}{qu_{1}},u_{2},\ldots u_{s},q^{1/2}u_{1}u_{s+1},\ldots,q^{1/2}u_{1}u_{m};M\right)\\
 & -\frac{qu_{1}+u_{1}-2}{n}\sum_{0\le j\le n-1}L_{\fudge}\left(\frac{1}{qu_{1}},u_{2},\ldots,u_{s},\zeta_{n}^{jM_{1,s+1}}q^{1/2}u_{1}u_{s+1},\ldots,\zeta_{n}^{jM_{1,m}}q^{1/2}u_{1}u_{m};M\right),
\end{align*}
which completes the proof of the functional equations as a formal
equality (Theorem \ref{thm:mainresult}).

\subsection{\label{subsec:A-short-verification}A short verification}

We explain how our functional equations recover the functional equations
in Whitehead's thesis, which, as we mentioned earlier, use the axioms
from Diaconu and Pasol's paper---this is completely analogous to
how we derive functional equations using the axioms from Sawin's paper.
One key difference is that Whitehead's argument is purely numerical
and elementary, whereas our argument is a combination of numerical
and geometric methods. 

Let us first recast the functional equations appearing as equations
(2.2.3) and (2.2.4) of Whitehead's thesis in terms of the notation
of our paper; we will freely use the notation of \cite{whitehead_multiple}
in this example. 

When Whitehead writes $j\sim i$, this is equivalent to the matrix
entry $M_{i,j}$ being equal to one. Otherwise, $M_{i,j}=0$. He also
requires $M_{i,i}=0$ for all $i$. Hence, without loss of generality,
in our setting where we require $s\ge1$, the symmetric matrix $M$
looks like 
\[
\begin{bmatrix}0 & 0 & \cdots & 0 & 1 & \cdots & 1\\
0 & 0 & * & * & * & * & *\\
\vdots & * & \ddots & * & * & * & *\\
0 & * & * & \ddots & * & * & *\\
1 & * & * & * & \ddots & * & *\\
\vdots & * & * & * & * & \ddots & *\\
1 & * & * & * & * & * & 0
\end{bmatrix},
\]
where the blank spots are arbitrary (1 or 0). 

Since $\chi$ is a non-trivial quadratic character in \cite{whitehead_multiple},
we have $n=2.$ 

Note that $M'=M$ because 
\begin{enumerate}
\item $M'_{i,j}=M_{i,j}+M_{1,i}+M_{1,j}=M_{i,j}+2\equiv M_{i,j}$ for $j>i\ge s+1,$ 
\item $M'_{i,i}=M_{i,i}+M_{1,i}+n/2=M_{i,i}+1+1\equiv M_{i,i}$ for $i\ge s+1$, 
\item $M'_{1,i}=-M_{1,i}\equiv M_{1,i}$ for all $i$,
\item $M_{i,j}'=M_{i,j}$ for $j>i>1$ and $i\le s$, and 
\item $M_{i,i}'=M_{i,i}$ for $i\le s$.
\end{enumerate}
In Whitehead's notation, $a_{i}=d_{i},$ $x_{i}=u_{i},$ and there
are two functional equations depending on the parity of $\sum_{j\sim1}a_{j}$. 

Since $q\equiv1\bmod4$ (by assumption in \cite{whitehead_multiple}),
we have 
\begin{align*}
\fudge\left(d_{2},\ldots,d_{m};M\right) & =\frac{\chi(-1)^{\sum_{s+1\le i<j\le m}d_{i}d_{j}M_{1,i}}(-1)^{\sum_{i\ge s+1}\frac{d_{i}(d_{i}-1)(q-1)}{4}}}{\prod_{i\ge s+1}G\left(\chi^{M_{1,i}},\psi\right)^{d_{i}}}\\
 & =\frac{1}{G\left(\chi,\psi\right)^{\sum_{i\ge s+1}d_{i}}}\\
 & =\frac{1}{\left(q^{1/2}\right)^{\sum_{i\ge s+1}d_{i}}}
\end{align*}
where we use the fact that the unique non-trivial quadratic character
$\chi\colon\F_{q}^{\times}\to\C^{\times}$ is given by $\left(\frac{\Nm_{\F_{q}/\F_{p}}\left(-\right)}{p}\right)$,
which means $\chi(-1)=1$ (if $p\equiv1\bmod4$, then $\left(\frac{-1}{p}\right)=1$
by quadratic reciprocity; if $p\equiv3\bmod4,$ then $q$ is necessarily
an even power of $p$), as well as Gauss's computation of a quadratic
Gauss sum.

The case $\sum_{j\sim1}a_{j}$ is odd means $\sum_{i\ge s+1}d_{i}$
is odd, i.e. the case where $\prod_{i=s+1}^{m}\left(\frac{-}{f_{i}}\right)_{\chi}^{M_{1,i}}$
is not trivial on $\F_{q}^{\times}$ from Subsection \ref{subsec:-is-trivial}. 

Then, Proposition \ref{prop:badlocalfunc} tells us that 

\begin{align*}
P_{f_{2},\ldots,f_{m};M}(u_{1}) & =\frac{\chi(-1)^{\sum_{i=s+1}^{m}d_{i}M_{1,i}}G\left(\chi^{\sum_{i=s+1}^{m}d_{i}M_{1,i}},\psi\right)}{q\fudge\left(d_{\ge s+1};M_{1,\ge s+1}\right)}u_{1}^{d-1}P_{f_{2},\ldots,f_{m};M'}\left(\frac{1}{qu_{1}}\right)\\
 & =\frac{G\left(\chi,\psi\right)\left(q^{1/2}\right)^{\sum_{i\ge s+1}d_{i}}}{q}u_{1}^{d-1}P_{f_{2},\ldots,f_{m};M'}\left(\frac{1}{qu_{1}}\right)\\
 & =\left(q^{1/2}u_{1}\right){}^{\left(\sum_{i\ge s}d_{i}\right)-1}P_{f_{2},\ldots,f_{m};M'}\left(q^{-1}u_{1}^{-1}\right).
\end{align*}

Summing both sides over all $f_{i}\in\mathcal{M}_{d_{i}}$ for $i\in\left\{ 2,\ldots,m\right\} $
gives the functional equation (2.2.3) of \cite{whitehead_multiple}.

The case $\sum_{j\sim i}a_{j}$ is even means $\sum_{i\ge s}d_{i}$
is even, i.e. the first case where $\prod_{i=s+1}^{m}\left(\frac{-}{f_{i}}\right)_{\chi}^{M_{1,i}}$
is trivial on $\F_{q}^{\times}$. Then, we know that 
\begin{align*}
P_{f_{2},\ldots,f_{m};M}\left(u_{1}\right)\left(qu_{1}-1\right) & =\frac{1}{\fudge\left(d_{\ge s+1};M_{1,\ge s+1}\right)}u_{1}^{d-1}\left(1-u_{1}\right)P_{f_{2},\ldots,f_{m};M'}\left(\frac{1}{qu_{1}}\right)\\
 & =\left(q^{1/2}u_{1}\right){}^{\sum_{i\ge s}d_{i}}\left(\frac{1-u_{1}}{u_{1}}\right)P_{f_{2},\ldots,f_{m};M'}\left(\frac{1}{qu_{1}}\right),
\end{align*}
so 
\[
\left(1-qu_{1}\right)P_{f_{2},\ldots,f_{m};M}(u_{1})=\left(q^{1/2}u_{1}\right){}^{\sum_{i\ge s+1}d_{i}}\left(1-u_{1}^{-1}\right)P_{f_{2},\ldots,f_{m};M'}\left(q^{-1}u_{1}^{-1}\right).
\]
Summing both sides over all $f_{i}\in\mathcal{M}_{d_{i}}$ for $i\in\left\{ 2,\ldots,m\right\} $
gives the functional equation (2.2.5) of \cite{whitehead_multiple}.

\section{Bounds on $a$-coefficients and their sums}

We now begin the second step of the proofs of the main results. Using
the notation in \cite{sawin_general}, along with notation of the
previous section, let 
\[
\lambda\left(d_{1},\ldots,d_{m};M\right)=\sum_{f_{1}\in\mathcal{M}_{d_{1}},\ldots,f_{m}\in\mathcal{M}_{d_{m}}}a\left(f_{1},\ldots,f_{m};M\right).
\]
In this section, we obtain upper bounds for $a\left(f_{1},\ldots,f_{m};M\right)$
and $\lambda\left(d_{1},\ldots,d_{m};M\right)$. To do so, we use
the Grothendieck--Lefschetz fixed point formula to bound $\lambda\left(d_{1},\ldots,d_{m};M\right)$
in terms of traces of Frobenius acting on compactly-supported cohomology
groups with coefficients in $K_{d_{1},\ldots,d_{m}}$. These, in turn,
can be bounded in terms of the dimensions of the cohomology groups.
These cohomology groups can be viewed as direct summands of cohomology
groups of a suitable compactification (via Kontsevich moduli spaces
of stable maps) with coefficients in a lisse rank one sheaf using
the decomposition theorem for perverse sheaves. Finally, we can use
the earlier results of Subsection \ref{subsec:generalbound}.

To bound $a\left(f_{1},\ldots,f_{m};M\right)$, we use the bound for
$\lambda\left(d_{1},\ldots,d_{m};M\right)$ along with the axioms
of a multiple Dirichlet series, namely the local-to-global relationship
and normalization (the last four axioms).

Let $r=d_{1}+\cdots+d_{m}$. 

We have 
\begin{align*}
\lambda\left(d_{1},\ldots,d_{m};M\right) & =\sum_{f_{1}\in\mathcal{M}_{d_{1}},\ldots,f_{m}\in\mathcal{M}_{d_{m}}}a\left(f_{1},\ldots,f_{m};M\right)\\
 & =\sum_{f_{1}\in\mathcal{M}_{d_{1}},\ldots,f_{m}\in\mathcal{M}_{d_{m}}}\sum_{i}\left(-1\right)^{i}\Tr\left(\Fr_{q},\mathcal{H}^{i}\left(K_{d_{1},\ldots,d_{m}}\right)_{\left(f_{1},\ldots,f_{m}\right)}\right)\\
 & =\sum_{i}\left(-1\right)^{i}\Tr\left(\Fr_{q},H_{c}^{i}\left(\prod_{j=1}^{m}\A_{\ol{\F_{q}}}^{d_{j}},K_{d_{1},\ldots,d_{m}}\right)\right),
\end{align*}
where we use the Grothendieck--Lefschetz fixed point formula (c.f.
\cite{Fu}) in the last equality.

Then, using the fact that $K_{d_{1},\ldots,d_{m}}$ is pure of weight
zero and Artin vanishing (c.f. \cite{Fu}), we obtain
\begin{align}
\left|\lambda\left(d_{1},\ldots,d_{m};M\right)\right| & \le\sum_{i=0}^{2r}\left|\Tr\left(\Fr_{q},H_{c}^{i}\left(\prod_{j=1}^{m}\A_{\ol{\F_{q}}}^{d_{j}},K_{d_{1},\ldots,d_{m}}\right)\right)\right|\le\sum_{i=0}^{2r}\left(\dim H_{c}^{i}\left(\prod_{j=1}^{m}\A_{\ol{\F_{q}}}^{d_{j}},K_{d_{1},\ldots,d_{m}}\right)\right)q^{i/2}.\label{eq:boundvialefschetz}
\end{align}
Note that the quotient map $\prod_{j=1}^{m}\A_{\ol{\F_{q}}}^{d_{j}}\to\left[\prod_{j=1}^{m}\A_{\ol{\F_{q}}}^{d_{j}}/S_{d_{1}}\times\cdots\times S_{d_{m}}\right]$
is proper and etale, and consequently $p\colon\left[\prod_{j=1}^{m}\A_{\ol{\F_{q}}}^{d_{j}}/S_{d_{1}}\times\cdots\times S_{d_{m}}\right]\to\prod_{j=1}^{m}\A_{\ol{\F_{q}}}^{d_{j}}/S_{d_{1}}\times\cdots\times S_{d_{m}}\cong\prod_{j=1}^{m}\A_{\ol{\F_{q}}}^{d_{j}}$
is proper and quasi-finite. Let $X$ denote the locus of tuples $\left(f_{1},\ldots,f_{m}\right)$
such that $\left(f_{i},f_{j}\right)=\left(f_{i},f_{i}'\right)=1$
for $i\ne j$ for all $1\le i,j\le m$, i.e. the pure configuration
space of $r$ points on $\A^{1}$ modulo $S_{d_{1}}\times\cdots\times S_{d_{m}}.$
Consider the following commutative diagram:\[\begin{tikzcd} 	{\left[\prod_{j=1}^{m}\A_{\ol{\F_{q}}}^{d_{j}}/S_{d_{1}}\times\cdots\times S_{d_{m}}\right]} & {\prod_{j=1}^{m}\A_{\ol{\F_{q}}}^{d_{j}}} \\ 	X 	\arrow["{j_\text{stack}}", hook, from=2-1, to=1-1] 	\arrow["j"', hook, from=2-1, to=1-2] 	\arrow["p", from=1-1, to=1-2] \end{tikzcd}\]

By the decomposition theorem for perverse sheaves from \cite{bbdg},
we have 
\begin{align*}
K_{d_{1},\ldots,d_{m}} & =j_{!*}\left(\mathcal{L}_{\chi}\left(F_{d_{1},\ldots,d_{m}}\right)\left[r\right]\right)\left[-r\right]\inplus p_{*}\left(j_{\text{stack}}\right)_{!*}\left(\mathcal{L}_{\chi}\left(F_{d_{1},\ldots,d_{m}}\right)\left[r\right]\right)\left[-r\right],
\end{align*}
where ``$\inplus$'' means direct summand. 

Let $X'$ be the pure configuration space of $r$ distinct points
on $\A^{1}$, and by abuse of notation let $j\colon X'\subset\prod_{j=1}^{m}\A_{\ol{\F_{q}}}^{d_{j}}$
be the natural open immersion. Then, by applying cohomology, we have
\begin{align}
H_{c}^{i}\left(\prod_{j=1}^{m}\A_{\ol{\F_{q}}}^{d_{j}},K_{d_{1},\ldots,d_{m}}\right) & \inplus H_{c}^{i}\left(\left[\prod_{j=1}^{m}\A_{\ol{\F_{q}}}^{d_{j}}/S_{d_{1}}\times\cdots\times S_{d_{m}}\right],\left(j_{\text{stack}}\right)_{!*}\left(\mathcal{L}_{\chi}\left(F_{d_{1},\ldots,d_{m}}\right)\left[r\right]\right)\left[-r\right]\right)\nonumber \\
 & =H_{c}^{i}\left(\prod_{j=1}^{m}\A_{\ol{\F_{q}}}^{d_{j}},j_{!*}\left(\mathcal{L}_{\chi}\left(F_{d_{1},\ldots,d_{m}}\right)\left[r\right]\right)\left[-r\right]\right)^{S_{d_{1}}\times\cdots\times S_{d_{m}}},\label{eq:stackvsgeo}
\end{align}

Recall (c.f. \cite{Kock_Vainsencher}) that there is a canonical evaluation
map $\ev\colon\mathcal{\ol M}_{0,r}\left(\P^{1},1\right)\to\left(\P^{1}\right)^{r}$
which sends a stable map to the images of its marked points. 

Then, consider the following Cartesian diagrams, where $\mathcal{\ol M}_{0,r}\left(\P^{1},1\right)_{U}$
denotes the pull-back of $\ev$ by the open inclusion $U\subset\left(\P^{1}\right)^{r}$:\[\begin{tikzcd} 	{\overline{\mathcal{M}}_{0,r}\left(\mathbb{P}^1,1\right)} & {\left(\mathbb{P}^1\right)^r} \\ 	{\overline{\mathcal{M}}_{0,r}\left(\mathbb{P}^1,1\right)_{\mathbb{A}^r}} & {\mathbb{A}^r} \\ 	{\overline{\mathcal{M}}_{0,r}\left(\mathbb{P}^1,1\right)_{X'}} & X' 	\arrow["\ev", from=1-1, to=1-2] 	\arrow[hook, from=2-2, to=1-2] 	\ar["\square", phantom, from=1-1, to=2-2] \arrow["\ev", from=2-1, to=2-2] 	\arrow[hook, from=2-1, to=1-1]  	\arrow[hook, from=3-2, to=2-2] 	\arrow[hook, from=3-1, to=2-1] 	\arrow["\ev", from=3-1, to=3-2] \ar["\square", phantom, from=2-1, to=3-2] \end{tikzcd}\]
\begin{lem}
\label{lem:ev_is_iso}$\ev\colon\ol{\mathcal{M}}_{0,r}\left(\P^{1},1\right)_{X'}\to X'$
is an isomorphism.
\begin{proof}
$\ol{\mathcal{M}}_{0,r}\left(\P^{1},1\right)$ is given by chains
of $\P^{1}$'s, along with a distinguished $\P^{1}$ that maps isomorphically
(of degree 1) onto $\P^{1}$. Moreover, every $\P^{1}$ that is not
the single distinguished $\P^{1}$ must be stable, i.e. have at least
three special points. But by assumption on $X$, this is only possible
if the twig is a single copy of $\P^{1}$, from which the result follows. 
\end{proof}
\end{lem}

Let the open immersion $\ol{\mathcal{M}}_{0,r}\left(\P^{1},1\right)_{X'}\cong X'\subset\ol{\mathcal{M}}_{0,r}\left(\P^{1},1\right)$
be denoted by $g$.

Note that $\ev$ is proper because $\ol{\mathcal{M}}_{0,r}\left(\P^{1},1\right)$
and $\left(\P^{1}\right)^{r}$ are proper.
\begin{lem}
\label{lem:restriction_to_open}$j^{*}R\ev_{*}g_{!*}\left(\mathcal{L}_{\chi}\left(F_{d_{1},\ldots,d_{m}}\right)\left[r\right]\right)\cong\mathcal{L}_{\chi}\left(F_{d_{1},\ldots,d_{m}}\right)\left[r\right]$. 
\begin{proof}
By proper base change (c.f. \cite{Fu}), the left-hand side is isomorphic
to $R\ev_{*}g^{*}g_{!*}\left(\mathcal{L}_{\chi}\left(F_{d_{1},\ldots,d_{m}}\right)\left[r\right]\right)$,
so the result follows from Lemma \ref{lem:ev_is_iso}. 
\end{proof}
\end{lem}

Lemma \ref{lem:restriction_to_open} and the decomposition theorem
applied to the proper map $\ev$ gives that $j_{!*}\left(\mathcal{L}_{\chi}\left(F_{d_{1},\ldots,d_{m}}\right)\left[r\right]\right)$
is a direct summand of $R\ev_{*}g_{!*}\left(\mathcal{L}_{\chi}\left(F_{d_{1},\ldots,d_{m}}\right)\left[r\right]\right)$.

As a result, we have 
\begin{equation}
H_{c}^{i}\left(\A^{r},j_{!*}\left(\mathcal{L}_{\chi}\left(F_{d_{1},\ldots,d_{m}}\right)\left[r\right]\right)\right)\inplus H_{c}^{i}\left(\ol{\mathcal{M}}_{0,r}\left(\P^{1},1\right)_{\A^{r}},g_{!*}\left(\mathcal{L}_{\chi}\left(F_{d_{1},\ldots,d_{m}}\right)\left[r\right]\right)\right).\label{eq:boundfromdecomp}
\end{equation}

Hence, combining (\ref{eq:stackvsgeo}) and (\ref{eq:boundfromdecomp})
after taking $S_{d_{1}}\times\cdots\times S_{d_{m}}$-invariants,
we obtain
\begin{equation}
\dim H_{c}^{i}\left(\prod_{j=1}^{m}\A_{\ol{\F_{q}}}^{d_{j}},K_{d_{1},\ldots,d_{m}}\right)\le\dim H_{c}^{i}\left(\ol{\mathcal{M}}_{0,r}\left(\P^{1},1\right)_{\A^{r}},g_{!*}\left(\mathcal{L}_{\chi}\left(F_{d_{1},\ldots,d_{m}}\right)\left[r\right]\right)\left[-r\right]\right)^{S_{d_{1}}\times\cdots\times S_{d_{m}}}.\label{eq:boundwrittenagain}
\end{equation}
Note that the complement of $X'=\mathcal{M}_{0,r}\left(\P^{1},1\right)$
in $\mathcal{\ol M}_{0,r}\left(\P^{1},1\right)_{\A^{r}}$ is given
by the union of divisors comprising two $\P^{1}$'s (with one distinguished
$\P^{1}$ mapping isomorphically onto $\P^{1}$ with degree one);
indeed, by \cite{Kontsevich}, this complement is a normal crossings
divisor. By Proposition \ref{prop:IC_is_sheaf}, we have 
\[
g_{!*}\left(\mathcal{L}_{\chi}\left(F_{d_{1},\ldots,d_{m}}\right)\left[r\right]\right)=g_{!}'g_{*}''\left(\mathcal{L}_{\chi}\left(F_{d_{1},\ldots,d_{m}}\right)\right)\left[r\right],
\]
where $g$ is the composition $g'\circ g''$ such that $g''$ is the
inclusion of $X'$ into the union of $X'$ and the divisors for which
the local monodromy of $\mathcal{L}_{\chi}\left(F_{d_{1},\ldots,d_{m}}\right)$
around each divisor is trivial, and $g'$ is the remaining inclusion
into $\mathcal{\ol M}_{0,r}\left(\P^{1},1\right)_{\A^{r}}$. 

Hence, (\ref{eq:boundwrittenagain}) can be written as 
\begin{equation}
\dim H_{c}^{i}\left(\prod_{j=1}^{m}\A_{\ol{\F_{q}}}^{d_{j}},K_{d_{1},\ldots,d_{m}}\right)\le\dim H_{c}^{i}\left(\ol{\mathcal{M}}_{0,r}\left(\P^{1},1\right)_{\A^{r}},g_{!}'g_{*}''\left(\mathcal{L}_{\chi}\left(F_{d_{1},\ldots,d_{m}}\right)\right)\right)^{S_{d_{1}}\times\cdots\times S_{d_{m}}}.\label{eq:last_step}
\end{equation}

\cite{EVW} prove \footnote{Their statement is for $L$ a constant sheaf, but the proof is identical.}
the following (as Proposition 7.7):
\begin{prop}
\label{prop:prop7.7}Let $A$ be a Henselian discrete valuation ring
whose quotient field has characteristic zero. Let $\ol{\eta}$ and
$\ol s$ be the generic and special points of $\Spec A$, respectively. 

Suppose $X$ is a scheme smooth and proper over $\Spec A$, $D\hookrightarrow X$
is a normal crossings divisor relative to $\Spec A$, and $U=X\backslash D$.
Let $G$ be a finite group acting on $X$ and $U$ compatibly and
$L$ a lisse sheaf on $U$. 

Then, 
\[
H_{c}^{i}\left(U_{\ol{\eta}},L\right)\cong H_{c}^{i}\left(U_{\ol s},L\right)
\]
 as $G$-modules for all $i$. 
\end{prop}

Let us apply this proposition to our situation. By setting $A$ to
be the Witt vectors $W\left(\F_{p}\right)$, and since $g_{*}''\left(\mathcal{L}_{\chi}\left(F_{d_{1},\ldots,d_{m}}\right)\right)$
is a lisse sheaf (by Proposition \ref{prop:IC_is_sheaf}), Proposition
\ref{prop:prop7.7} tells us that we may actually work in the characteristic
zero situation.

To bound 
\begin{align*}
 & \dim H_{c}^{i}\left(\ol{\mathcal{M}}_{0,r}\left(\P^{1},1\right)_{\A^{r}},g_{!}'g_{*}''\left(\mathcal{L}_{\chi}\left(F_{d_{1},\ldots,d_{m}}\right)\right)\right)^{S_{d_{1}}\times\cdots\times S_{d_{m}}}\\
 & =\dim H_{c}^{i}\left(\left[\ol{\mathcal{M}}_{0,r}\left(\P^{1},1\right)_{\A^{r}}/S_{d_{1}}\times\cdots\times S_{d_{m}}\right],g_{!}'g_{*}''\left(\mathcal{L}_{\chi}\left(F_{d_{1},\ldots,d_{m}}\right)\right)\right)\\
 & =\dim H_{c}^{i}\left(\left[\left(\ol{\mathcal{M}}_{0,r}\left(\P^{1},1\right)_{\A^{r}}\backslash D\right)/S_{d_{1}}\times\cdots\times S_{d_{m}}\right],g_{*}''\left(\mathcal{L}_{\chi}\left(F_{d_{1},\ldots,d_{m}}\right)\right)\right),
\end{align*}
where $D$ is the union of divisors for which the local monodromy
of $\mathcal{L}_{\chi}\left(F_{d_{1},\ldots,d_{m}}\right)$ is non-trivial,
Poincare duality (c.f. \cite{Fu}) implies that it suffices to bound
the more general quantity

\[
\dim H^{*}\left(\left[\left(\ol{\mathcal{M}}_{0,r}\left(\P^{1},1\right)_{\A^{r}}\backslash D\right)/S_{d_{1}}\times\cdots\times S_{d_{m}}\right],\mathcal{L}\right),
\]
for any rank one lisse sheaf $\mathcal{L}$ on $\left[\left(\ol{\mathcal{M}}_{0,r}\left(\P^{1},1\right)_{\A^{r}}\backslash D\right)/S_{d_{1}}\times\cdots\times S_{d_{m}}\right],$
which we already did in Subsection \ref{subsec:generalbound}.

Let $C=64m$. Combining (\ref{eq:last_step}), Proposition \ref{prop:prop7.7},
and Proposition \ref{prop:cohomologybound}, we obtain the bound 
\begin{equation}
\dim H_{c}^{*}\left(\prod_{j=1}^{m}\A_{\ol{\F_{q}}}^{d_{j}},K_{d_{1},\ldots,d_{m}}\right)\ll_{m}C^{r}.\label{eq:boundoncj}
\end{equation}

Applying this to (\ref{eq:boundvialefschetz}), we get the following:
\begin{prop}
\label{prop:boundsonsum}
\begin{align*}
\left|\lambda\left(d_{1},\ldots,d_{m};M\right)\right| & \ll_{m}C^{r}\sum_{i=0}^{2r}q^{i/2}\ll_{q,m}\left(Cq\right)^{r}.
\end{align*}
\end{prop}

Let us now bound $a\left(f_{1},\ldots,f_{m};M\right)$ for $f_{i}\in\mathcal{M}_{d_{i}}$. 
\begin{cor}
\label{cor:boundsona}Let $\pi$ be a prime. Then, for $\sum_{i}e_{i}=0$
or 1, we have $\left|a\left(\pi^{e_{1}},\ldots,\pi^{e_{m}}\right)\right|=1$,
and for $\sum_{i}e_{i}\ge2$, we have 
\[
\left|a\left(\pi^{e_{1}},\ldots,\pi^{e_{m}}\right)\right|\ll_{m}C^{\ep+\sum_{i}e_{i}}q^{-\deg\pi}q^{\frac{\sum_{i}e_{i}}{2}\deg\pi}
\]
for all $\ep>0$.
\begin{proof}
If $\sum_{i}e_{i}=0$ or 1, $K_{e_{1},\ldots,e_{m}}$ is the constant
sheaf. Then, by the third axiom of Theorem \ref{thm:axioms}, we have
\[
\left|a\left(\pi^{e_{1}},\ldots,\pi^{e_{m}}\right)\right|=\left|\sum_{j\in J\left(e_{1},\ldots,e_{m};q,\chi,M\right)}c_{j}\alpha_{j}^{\deg\pi}\right|=1.
\]

For $\sum_{i}e_{i}\ge2$, by the fifth axiom of Theorem \ref{thm:axioms},
\begin{align*}
\left|a\left(\pi^{e_{1}},\ldots,\pi^{e_{m}}\right)\right| & =\left|\sum_{j\in J\left(e_{1},\ldots,e_{m};q,\chi,M\right)}c_{j}\alpha_{j}^{\deg\pi}\right|\le\left(\sum_{j\in J\left(e_{1},\ldots,e_{m};q,\chi,M\right)}\left|c_{j}\right|\right)\left(q^{\frac{\sum_{i}e_{i}}{2}-1}\right)^{\deg\pi}.
\end{align*}
By Theorem \ref{thm:axioms}, $c_{j}$ is the signed multiplicity
of $\alpha_{j},$ which is an eigenvalue of $\Fr_{q}$ acting on the
complex $\left(K_{d_{1},\ldots,d_{m}}\right)_{\left(t^{d_{1}},\ldots,t^{d_{m}}\right)}$.
By $\G_{m}$-localization (Lemma 2.16 of \cite{sawin_general}), Artin
vanishing (c.f. \cite{Fu}), and (\ref{eq:boundoncj}), we have 
\begin{align*}
\sum_{j\in J\left(e_{1},\ldots,e_{m};q,\chi,M\right)}\left|c_{j}\right| & =\sum_{i}\dim\mathcal{H}^{i}\left(\left(K_{e_{1},\ldots,e_{m}}\right)_{\left(T^{e_{1}},\ldots,T^{e_{m}}\right)}\right)\\
 & =\sum_{i}\dim H_{c}^{i}\left(\prod\A_{\ol{\F_{q}}}^{e_{j}},K_{e_{1},\ldots,e_{m}}\right)\\
 & =\sum_{i=0}^{2(e_{1}+\cdots+e_{m})}\dim H_{c}^{i}\left(\prod\A_{\ol{\F_{q}}}^{e_{j}},K_{e_{1},\ldots,e_{m}}\right)\\
 & \ll_{m}C^{\sum_{i}e_{i}+\ep}
\end{align*}
for all $\ep>0$. Then, for all $\ep>0$, we have 
\[
\left|a\left(\pi^{e_{1}},\ldots,\pi^{e_{m}}\right)\right|\ll_{m}C^{\ep+\sum_{i}e_{i}}\left(q^{\frac{\sum_{i}e_{i}}{2}-1}\right)^{\deg\pi}=C^{\ep+\sum_{i}e_{i}}q^{-\deg\pi}q^{\frac{\sum_{i}e_{i}}{2}\deg\pi},
\]
as desired.
\end{proof}
\end{cor}

\section{Meromorphic continuation of multiple Dirichlet series}

Using the bounds of the previous section, we are finally able to finish
the proofs of Theorems \ref{thm:firstmain} and \ref{thm:mainresult}.
This is 

Again, let us continue using the notation of the previous section.
In particular, recall that $C=64m$. 

First, we show the following:
\begin{prop}
\label{prop:analyticeq}$\left(qu_{1}-1\right)L\left(u_{1},\ldots,u_{m};M\right)$
and $\left(qu_{1}-1\right)L_{\fudge}\left(u_{1},\ldots,u_{m};M\right)$
converge for 
\[
\left|u_{s+1}\right|,\ldots,\left|u_{m}\right|<\frac{1}{Cq\max\left\{ Cq\left|u_{1}\right|,1\right\} };\left|u_{2}\right|,\ldots,\left|u_{s}\right|<\frac{1}{Cq}.
\]
\end{prop}

\begin{rem}
Note that the inequalities 
\[
\left|u_{s+1}\right|,\ldots,\left|u_{m}\right|<\frac{1}{Cq\max\left\{ Cq\left|u_{1}\right|,1\right\} };\left|u_{2}\right|,\ldots,\left|u_{s}\right|<\frac{1}{Cq}
\]
and 
\[
\left|q^{1/2}u_{1}u_{s+1}\right|,\ldots,\left|q^{1/2}u_{1}u_{m}\right|<\frac{1}{Cq\max\left\{ Cq\left|\frac{1}{qu_{1}}\right|,1\right\} };\left|u_{2}\right|,\ldots,\left|u_{s}\right|<\frac{1}{Cq}
\]
are simultaneously satisfied by 
\[
u_{1}\ne0;\left|u_{1}\right|,\ldots,\left|u_{s}\right|<\frac{1}{Cq};\left|u_{s+1}\right|,\ldots,\left|u_{m}\right|<\frac{1}{C^{2}q^{3/2}};\left|u_{2}\right|,\ldots,\left|u_{s}\right|<\frac{1}{Cq}.
\]
Hence, the proposition completes the proofs of Theorems \ref{thm:firstmain}
and \ref{thm:mainresult}.

Beyond the range of convergence achieved in Proposition \ref{prop:analyticeq},
we show in Proposition \ref{prop:extraconverge} below how to extend
the range of convergence slightly further. 
\end{rem}

\begin{proof}
[Proof of Proposition \ref{prop:analyticeq}] Let us prove convergence
for $\left(qu_{1}-1\right)L\left(u_{1},\ldots,u_{m};M\right)$ first. 

From Lemma \ref{lem:expressionforP}, we have 
\[
\left(qu_{1}-1\right)P_{f_{2},\ldots,f_{m};M}\left(u_{1}\right)=-S_{d_{s+1}+\cdots+d_{m};f_{2},\ldots,f_{m};M}u_{1}^{d_{s+1}+\cdots+d_{m}}+\sum_{t=0}^{d_{s+1}+\cdots+d_{m}-1}\left(qu_{1}-1\right)S_{t;f_{2},\ldots,f_{m};M}u_{1}^{t}.
\]
Summing over all $f_{2}\in\mathcal{M}_{d_{2}},\ldots,f_{m}\in\mathcal{M}_{d_{m}}$,
we obtain 
\begin{align*}
\left(qu_{1}-1\right)\sum_{f_{i}\in\mathcal{M}_{d_{i}}}P_{f_{2},\ldots,f_{m};M}\left(u_{1}\right) & =\sum_{t=0}^{d_{s+1}+\cdots+d_{m}}\left(q\lambda\left(t-1,d_{2},\ldots,d_{m};M\right)-\lambda\left(t,d_{2},\ldots,d_{m};M\right)\right)u_{1}^{t}.
\end{align*}

By Proposition \ref{prop:boundsonsum}, we have
\begin{align*}
\left|q\lambda\left(t-1,d_{2},\ldots,d_{m};M\right)-\lambda\left(t,d_{2},\ldots,d_{m};M\right)\right| & \le q\left|\lambda\left(t-1,d_{2},\ldots,d_{m};M\right)\right|+\left|\lambda\left(t,d_{2},\ldots,d_{m};M\right)\right|\\
 & \ll_{q,m}\left(Cq\right)^{t+d_{2}+\cdots+d_{m}}.
\end{align*}
It follows that
\begin{align*}
 & \left|\left(qu_{1}-1\right)\sum_{f_{2}\in\mathcal{M}_{d_{2}},\ldots,f_{m}\in\mathcal{M}_{d_{m}}}P_{f_{2},\ldots,f_{m};M}\left(u_{1}\right)\right|\\
 & \le\sum_{t=0}^{d_{s+1}+\cdots+d_{m}}\left|q\lambda\left(t-1,d_{2},\ldots,d_{m};M\right)-\lambda\left(t,d_{2},\ldots,d_{m};M\right)\right|\left|u_{1}\right|^{t}\\
 & \ll_{q,m}\sum_{t=0}^{d_{s+1}+\cdots+d_{m}}\left(Cq\right)^{d_{2}+\cdots+d_{m}}\left|Cqu_{1}\right|^{t}\\
 & \ll_{q,m}\left(Cq\right)^{d_{2}+\cdots+d_{m}}\left(1+d_{s+1}+\cdots+d_{m}\right)\max\left\{ 1,\left|Cqu_{1}\right|^{d_{s+1}+\cdots+d_{m}}\right\} .
\end{align*}
Now, recall 
\[
L\left(u_{1},\ldots,u_{m};M\right)=\sum_{d_{2},\ldots,d_{m}}\sum_{f_{2}\in\mathcal{M}_{d_{2}},\ldots,f_{m}\in\mathcal{M}_{d_{m}}}P_{f_{2},\ldots,f_{m};M}(u_{1})u_{2}^{d_{2}}\cdots u_{m}^{d_{m}},
\]
so 
\begin{align*}
 & \left|\left(qu_{1}-1\right)L\left(u_{1},\ldots,u_{m};M\right)\right|\\
 & \le\sum_{d_{2},\ldots,d_{m}}\left|\left(qu_{1}-1\right)\sum_{f_{2}\in\mathcal{M}_{d_{2}},\ldots,f_{m}\in\mathcal{M}_{d_{m}}}P_{f_{2},\ldots,f_{m};M}(u_{1})u_{2}^{d_{2}}\cdots u_{m}^{d_{m}}\right|\\
 & \ll_{q,m}\sum_{d_{2},\ldots,d_{m}}\left(Cq\right)^{d_{2}+\cdots+d_{m}}\left(1+d_{s+1}+\cdots+d_{m}\right)\max\left\{ 1,\left|Cqu_{1}\right|^{d_{s+1}+\cdots+d_{m}}\right\} \left|u_{2}\right|^{d_{2}}\cdots\left|u_{m}\right|^{d_{m}}\\
 & =\sum_{d_{2},\ldots,d_{m}}\left(1+d_{s+1}+\cdots+d_{m}\right)\max\left\{ 1,\left|Cqu_{1}\right|^{d_{s+1}+\cdots+d_{m}}\right\} \left|Cqu_{2}\right|^{d_{2}}\cdots\left|Cqu_{m}\right|^{d_{m}}\\
 & =\sum_{d_{2},\ldots,d_{m}}\left(1+d_{s+1}+\cdots+d_{m}\right)\left|Cqu_{2}\right|^{d_{2}}\cdots\left|Cqu_{s}\right|^{d_{s}}\left|\max\left\{ 1,\left|Cqu_{1}\right|\right\} Cqu_{s+1}\right|^{d_{s+1}}\\
 & \qquad\cdots\left|\max\left\{ 1,\left|Cqu_{1}\right|\right\} Cqu_{m}\right|^{d_{m}}\\
 & =\left(\prod_{i=2}^{s}\frac{1}{1-Cq\left|u_{i}\right|}\right)\sum_{d_{s+1},\ldots,d_{m}}\left(1+d_{s+1}+\cdots+d_{m}\right)\left|\max\left\{ 1,\left|Cqu_{1}\right|\right\} Cqu_{s+1}\right|^{d_{s+1}}\\
 & \qquad\cdots\left|\max\left\{ 1,\left|Cqu_{1}\right|\right\} Cqu_{m}\right|^{d_{m}},
\end{align*}
where in the last step we use the assumption that $\left|u_{2}\right|,\ldots,\left|u_{s}\right|<\frac{1}{Cq}$. 

It suffices to show 
\[
\sum_{d_{s+1},\ldots,d_{m}}\left(1+d_{s+1}+\cdots+d_{m}\right)\left|\max\left\{ 1,\left|Cqu_{1}\right|\right\} Cqu_{s+1}\right|^{d_{s+1}}\cdots\left|\max\left\{ 1,\left|Cqu_{1}\right|\right\} Cqu_{m}\right|^{d_{m}}
\]
 converges. To see this, we can write the expression as 
\begin{align*}
 & \sum_{d_{s+1},\ldots,d_{m}}\prod_{j=s+1}^{m}\left|\max\left\{ 1,\left|Cqu_{1}\right|\right\} Cqu_{j}\right|^{d_{j}}+\sum_{i=s+1}^{m}\sum_{d_{s+1},\ldots,d_{m}}d_{i}\prod_{j=s+1}^{m}\left|\max\left\{ 1,\left|Cqu_{1}\right|\right\} Cqu_{j}\right|^{d_{j}}\\
 & =\prod_{j=s+1}^{m}\frac{1}{1-\left|\max\left\{ 1,\left|Cqu_{1}\right|\right\} Cqu_{j}\right|}\\
 & \qquad+\sum_{i=s+1}^{m}\frac{\left|\max\left\{ 1,\left|Cqu_{1}\right|\right\} Cqu_{i}\right|}{1-\left|\max\left\{ 1,\left|Cqu_{1}\right|\right\} Cqu_{i}\right|}\prod_{j=s+1}^{m}\frac{1}{1-\left|\max\left\{ 1,\left|Cqu_{1}\right|\right\} Cqu_{j}\right|},
\end{align*}
which evidently converges for $\left|\max\left\{ 1,\left|Cqu_{1}\right|\right\} Cqu_{j}\right|<1$
for $j\in\left\{ s+1,\ldots,m\right\} $. 

For $L_{\fudge}(u_{1},\ldots,u_{m};M)$, note that the argument for
$L\left(u_{1},\ldots,u_{m};M\right)$ works without any modification
since $\left|b\left(d_{\ge s+1};M_{1,\ge s+1}\right)\right|$ is bounded
independently of $d_{\ge s+1}$. 
\end{proof}
Finally, we demonstrate how to obtain a different region of convergence
using the bounds on $a$-coefficients. In particular, this region
neither contains nor is contained in the region defined in Proposition
\ref{prop:analyticeq}, so meromorphic continuation allows us to extend
the domain of definition of the multiple Dirichlet series.
\begin{prop}
\label{prop:extraconverge}$L\left(u_{1},\ldots,u_{m};M\right)$ and
$L_{\fudge}(u_{1},\ldots,u_{m};M)$ converge for 
\[
\left|u_{i}\right|<\min\left\{ q^{-1},C^{-1}q^{-1/2}\right\} .
\]
\begin{proof}
Let us prove convergence for $L\left(u_{1},\ldots,u_{m};M\right)$
first. 

To get bounds on radii of convergence, it suffices to obtain radii
of convergence for the Euler product 
\[
\prod_{\pi\text{ prime}}\left(\sum_{e_{1},\ldots,e_{m}}\left|a\left(\pi^{e_{1}},\ldots,\pi^{e_{m}};M\right)\right|u_{1}^{e_{1}\deg\pi}\cdots u_{m}^{e_{m}\deg\pi}\right).
\]

Indeed, observe that showing convergence for 
\[
L\left(u_{1},\ldots,u_{m};M\right)=\sum_{f_{1},\ldots,f_{m}\in\F_{q}[t]^{+}}a\left(f_{1},\ldots,f_{m};M\right)u_{1}^{\deg f_{1}}\cdots u_{m}^{\deg f_{m}}
\]
is implied by convergence for 
\[
\sum_{f_{1},\ldots,f_{m}\in\F_{q}[t]^{+}}\left|a\left(f_{1},\ldots,f_{m};M\right)\right|u_{1}^{\deg f_{1}}\cdots u_{m}^{\deg f_{m}},
\]
and twisted multiplicativity gives
\[
\left|a\left(f_{1},\ldots,f_{m};M\right)\right|=\prod_{\pi\text{ prime}}\left|a\left(\pi^{v_{\pi}\left(f_{1}\right)},\ldots,\pi^{v_{\pi}\left(f_{m}\right)}\right)\right|.
\]
By Corollary \ref{cor:boundsona}, the inner term 
\begin{align*}
 & \sum_{e_{1},\ldots,e_{m}}\left|a\left(\pi^{e_{1}},\ldots,\pi^{e_{m}};M\right)\right|\left|u_{1}\right|^{e_{1}\deg\pi}\cdots\left|u_{m}\right|^{e_{m}\deg\pi}\\
 & =\sum_{\left(e_{1},e_{2},\ldots,e_{m}\right)=0}\left|a\left(\pi^{e_{1}},\ldots,\pi^{e_{m}};M\right)\right|\left|u_{1}\right|^{e_{1}\deg\pi}\cdots\left|u_{m}\right|^{e_{m}\deg\pi}\\
 & \qquad+\sum_{e_{1}+\cdots+e_{m}=1}\left|a\left(\pi^{e_{1}},\ldots,\pi^{e_{m}};M\right)\right|\left|u_{1}\right|^{e_{1}\deg\pi}\cdots\left|u_{m}\right|^{e_{m}\deg\pi}\\
 & \qquad+\sum_{e_{1}+\cdots+e_{m}\ge2}\left|a\left(\pi^{e_{1}},\ldots,\pi^{e_{m}};M\right)\right|\left|u_{1}\right|^{e_{1}\deg\pi}\cdots\left|u_{m}\right|^{e_{m}\deg\pi}\\
 & \ll_{m}1+\left|u_{1}\right|^{\deg\pi}+\cdots+\left|u_{m}\right|^{\deg\pi}+\sum_{\underset{\left(e_{2},\ldots,e_{m}\right)\ne0}{e_{1}+\cdots+e_{m}\ge2}}C^{\ep+\sum_{i}e_{i}}q^{-\deg\pi}q^{\frac{\sum_{i}e_{i}}{2}\deg\pi}\left|u_{1}\right|^{e_{1}\deg\pi}\cdots\left|u_{m}\right|^{e_{m}\deg\pi}.
\end{align*}
To establish absolute convergence of 
\[
\prod_{\pi\text{ prime}}\left(\sum_{e_{1},\ldots,e_{m}}\left|a\left(\pi^{e_{1}},\ldots,\pi^{e_{m}};M\right)\right|u_{1}^{e_{1}\deg\pi}\cdots u_{m}^{e_{m}\deg\pi}\right),
\]
it suffices (and is equivalent) to establishing absolute convergence
of 
\[
\sum_{\pi\text{ prime}}\left(\sum_{e_{1},\ldots,e_{m}}\left|a\left(\pi^{e_{1}},\ldots,\pi^{e_{m}};M\right)\right|u_{1}^{e_{1}\deg\pi}\cdots u_{m}^{e_{m}\deg\pi}-1\right).
\]
By the above, this is at most 
\begin{align*}
 & \ll_{m}\sum_{\pi\text{ prime}}\left(\left|u_{1}\right|^{\deg\pi}+\cdots+\left|u_{m}\right|^{\deg\pi}+q^{-\deg\pi}\sum_{e_{1}+\cdots+e_{m}\ge2}C^{\ep+\sum_{i}e_{i}}q^{\frac{\sum_{i}e_{i}}{2}\deg\pi}\left|u_{1}\right|^{e_{1}\deg\pi}\cdots\left|u_{m}\right|^{e_{m}\deg\pi}\right)\\
 & \le\sum_{d\ge1}\left(q^{d}\left|u_{1}\right|^{d}+\cdots+q^{d}\left|u_{m}\right|^{d}+\sum_{e_{1}+\cdots+e_{m}\ge2}C^{\ep+\sum_{i}e_{i}}q^{\frac{\sum_{i}e_{i}}{2}d}\left|u_{1}\right|^{e_{1}d}\cdots\left|u_{m}\right|^{e_{m}d}\right)\\
 & =\sum_{d\ge1}\left(q|u_{1}|\right)^{d}+\cdots+\sum_{d\ge1}\left(q|u_{m}|\right)^{d}+\sum_{d\ge1}\sum_{e_{1}+\cdots+e_{m}\ge2}C^{\ep+\sum_{i}e_{i}}q^{\frac{\sum_{i}e_{i}}{2}d}\left|u_{1}\right|^{e_{1}d}\cdots\left|u_{m}\right|^{e_{m}d}.
\end{align*}
First, note that 
\[
\sum_{d\ge1}\left(q|u_{i}|\right)^{d}=\frac{q|u_{i}|}{1-q|u_{i}|}
\]
converges because $\left|u_{i}\right|<1/q$. 

Next, 
\begin{align*}
\sum_{d\ge1}\sum_{e_{1}+\cdots+e_{m}\ge2}C^{\ep+\sum_{i}e_{i}}q^{\frac{\sum_{i}e_{i}}{2}d}\left|u_{1}\right|^{e_{1}d}\cdots\left|u_{m}\right|^{e_{m}d} & =\sum_{e_{1}+\cdots+e_{m}\ge2}C^{\ep+\sum_{i}e_{i}}\frac{q^{\frac{\sum_{i}e_{i}}{2}}\left|u_{1}\right|^{e_{1}}\cdots\left|u_{m}\right|^{e_{m}}}{1-q^{\frac{\sum_{i}e_{i}}{2}}\left|u_{1}\right|^{e_{1}}\cdots\left|u_{m}\right|^{e_{m}}}
\end{align*}
converges because $\left|u_{i}\right|<1/q$. This is moreover at most
\[
\sum_{e\ge2}{e+m-1 \choose m-1}C^{\ep+e}\frac{q^{e/2}\max_{1\le i\le m}\left\{ \left|u_{i}\right|\right\} ^{e}}{1-q^{e/2}\max_{1\le i\le m}\left\{ \left|u_{i}\right|\right\} ^{e}},
\]
which converges absolutely because of the ratio test:
\begin{align*}
\frac{{e+m \choose m-1}C^{\ep+e+1}\frac{q^{(e+1)/2}\max_{1\le i\le m}\left\{ \left|u_{i}\right|\right\} ^{e+1}}{1-q^{(e+1)/2}\max_{1\le i\le m}\left\{ \left|u_{i}\right|\right\} ^{e+1}}}{{e+m-1 \choose m-1}C^{\ep+e}\frac{q^{e/2}\max_{1\le i\le m}\left\{ \left|u_{i}\right|\right\} ^{e}}{1-q^{e/2}\max_{1\le i\le m}\left\{ \left|u_{i}\right|\right\} ^{e}}} & =\frac{e+m}{e}Cq^{1/2}\max_{1\le i\le m}\left\{ \left|u_{i}\right|\right\} \frac{1-q^{e/2}\max_{1\le i\le m}\left\{ \left|u_{i}\right|\right\} ^{e}}{1-q^{(e+1)/2}\max_{1\le i\le m}\left\{ \left|u_{i}\right|\right\} ^{e+1}}\\
 & \to Cq^{1/2}\max_{1\le i\le m}\left\{ \left|u_{i}\right|\right\} \\
 & <1,
\end{align*}
because $\left|u_{i}\right|<1/q$ for the second step and $\left|u_{i}\right|<C^{-1}q^{-1/2}$
for the last step.

For $L_{\fudge}(u_{1},\ldots,u_{m};M)$, note that the argument for
$L\left(u_{1},\ldots,u_{m};M\right)$ works without any modification
since $\left|b\left(d_{\ge s+1};M_{1,\ge s+1}\right)\right|$ is bounded
independently of $d_{\ge s+1}$. 
\end{proof}
\end{prop}

\printbibliography[heading=bibintoc]

\end{document}